\documentclass[12pt,reqno,a4paper]{amsart}
\usepackage{
    amsmath,  amsfonts, amssymb,  amsthm,   amscd,
    gensymb,  graphicx, comment,  etoolbox, url,
    booktabs, stackrel, mathtools,enumitem, mathdots,  microtype, lmodern,    mathrsfs, graphicx, tikz,  longtable, tabularx, float, tikz, pst-node, tikz-cd, multirow, tabularx, amscd,  bm, array, makecell, diagbox, booktabs,ragged2e, caption, subcaption }
\usepackage{makecell,slashbox}

\usepackage{comment}

\allowdisplaybreaks

\usepackage{xcolor}
\usepackage[utf8]{inputenc}
\usepackage{microtype, fullpage, wrapfig,textcomp,mathrsfs,csquotes,fbb}
\usepackage[colorlinks=true, linkcolor=blue, citecolor=blue, urlcolor=blue, breaklinks=true]{hyperref}
\usepackage[capitalise]{cleveref}
\usepackage{todonotes}
\usetikzlibrary{positioning}
\usetikzlibrary{shapes,arrows.meta,calc}

\usetikzlibrary{arrows}

\newtheorem{theorem}{Theorem}[section]

\newtheorem{corollary}[theorem] {Corollary}
\newtheorem{definition}[theorem]{Definition}
\newtheorem{example}[theorem]{Example}
\newtheorem{lemma}[theorem]{Lemma}

\newtheorem{proposition}[theorem]{Proposition}
\newtheorem{remark}[theorem]{Remark}

\newcommand\R{\mathbb{R}}
\newcommand\Z{\mathbb{Z}}

\newcommand{\TC}{\mathrm{TC}}

\newcommand{\ct}{\mathrm{cat}}
\newcommand{\zl}{\mathrm{zcl}}
\newcommand{\sct}{\mathrm{secat}}
\newcommand{\secat}{\mathrm{secat}}
\newcommand{\nil}{\mathrm{nil}}

\newcolumntype{x}[1]{>{\centering\arraybackslash}p{#1}}

\begin{document}
\title[]{Sectional category with respect to group actions and sequential topological complexity of fibre bundles}

\author[R. Singh]{Ramandeep Singh Arora}
\address{Department of Mathematics, Indian Institute of Science Education and Research Pune, India}
\email{ramandeepsingh.arora@students.iiserpune.ac.in}
\email{ramandsa@gmail.com}

\author[N. Daundkar]{Navnath Daundkar}
\address{Department of Mathematics, Indian Institute of Science Education and Research Pune, India.}
\email{navnath.daundkar@acads.iiserpune.ac.in}

\author[S. Sarkar]{Soumen Sarkar}
\address{Department of Mathematics, Indian Institute of Technology Madras, India.}
\email{soumen@iitm.ac.in}


\begin{abstract}
Let $X$ be a $G$-space. In this paper, we introduce the notion of sectional category with respect to $G$. As a result, we obtain $G$-homotopy invariants: the LS category with respect to $G$, the sequential topological complexity with respect to $G$ (which is same as the weak sequential equivariant topological complexity $\mathrm{TC}_{k,G}^w(X)$ in the sense of Farber and Oprea), and the strong sequential topological complexity with respect to $G$, denoted by $\mathrm{cat}_G^{\#}(X)$, $\mathrm{TC}_{k,G}^{\#}(X)$, and $\mathrm{TC}_{k,G}^{\#,*}(X)$, respectively. We explore several relationships among these invariants and well-known ones, such as the (equivariant) LS category, the sequential (equivariant) topological complexity, and the sequential strong equivariant topological complexity. In one of our main results, we give an additive upper bound for $\mathrm{TC}_k(E)$ for a fibre bundle $F \hookrightarrow E \to B$ with structure group $G$ in terms of certain motion planning covers of the base $B$ and the invariant $\mathrm{TC}_{k,G}^{\#,*}(F)$ or $\mathrm{cat}_{G^k}^{\#}(F^k)$, where the fibre $F$ is viewed as a $G$-space. As applications of these results, we give bounds on the LS category and the sequential topological complexity of generalized projective product spaces and mapping tori.

\end{abstract}

\keywords{LS category, sequential topological complexity, weak equivariant topological complexity, fibre bundles, sectional category, generalized projective product spaces}
\subjclass[2020]{55M30, 55R91, 55S40}
\maketitle

\section{Introduction}
Suppose that $B$ is a $G$-space and $p \colon E \to B$ is a fibration. We define the \textit{sectional category of $p$ with respect to $G$}, denoted by $\secat_G^{\#}(p)$, to be the smallest number $n$ such that there exists a $G$-invariant open cover $\{U_i\}^{n}_{i=1}$ of the base $B$ such that over each $U_i$ there exists a continuous section $s_i$ of $p$. 
This is a fibre-homotopy invariant which bounds the sectional category of $p$ from above, and if $p$ is a $G$-fibration, then it is bounded above by equivariant sectional category.

As a result, we obtain new numerical homotopy invariants for a $G$-space $X$, namely the LS category with respect to $G$, the sequential topological complexity with respect to $G$, and the strong sequential topological complexity with respect to $G$, denoted by $\ct_G^{\#}(X)$, $\TC_{k,G}^{\#}(X)$, and $\TC_{k,G}^{\#,*}(X)$, respectively.
More precisely, if $P_{x_0}X$ is the based path space of $(X,x_0)$ equipped with compact-open topology, i.e.,
$$
    P_{x_0}X = \{\alpha \colon I \to X \mid \alpha(0) = x_0\},
$$
and $e_X\colon P_{x_0}X \to X$ is the path space fibration defined by $e_X(\alpha)=\alpha(1)$, then $\ct_G^{\#}(X) := \secat_G^{\#}(e_X)$. 
Let $X^I$ be the free path space of $X$ with compact open topology.  Consider the generalized free path space fibration $e_{k,X}\colon X^I\to X^k$ defined by 
$$
e_{k,X}(\alpha) = \left(\alpha(0), \alpha\left(\frac{1}{k-1}\right), \ldots, \alpha\left(\frac{i}{k-1}\right), \ldots, \alpha\left(\frac{k-2}{k-1}\right), \alpha(1)\right),
$$ 
then we define $\TC_{k,G}^{\#}(X) := \secat_{G}^{\#}(e_{k,X})$ for the diagonal $G$-action on $X^k$, and $\TC_{k,G}^{\#,*}(X) := \secat_{G^k}^{\#}(e_{k,X})$ for the componentwise $G^k$-action on $X^k$.


In \Cref{sec:secat-related-invariants}, we study the various properties of these invariants. 
In \Cref{prop:cat-wrtG-properties}, \Cref{cor: cat-wrt-G leq equi-cat} and \Cref{prop:new-TC-properties}, we list several inequalities relating $\ct_G^{\#}(X)$, $\TC_{k,G}^{\#}(X)$, $\TC_{k,G}^{\#,*}(X)$ to their respective non-equivariant counterparts $\ct(X)$ and $\TC_k(X)$, and equivariant counterparts $\ct_G(X)$, $\TC_{k,G}(X)$ and $\TC_{k,G}^{*}(X)$, see \Cref{sec:preliminaries} and \Cref{rem_equitc} for definitions. 
Then we prove the homotopy invariance of these invariants in \Cref{thm:invariance-cat-wrt-G} and \Cref{thm:invariance-tck-wrt-G}, and obtain their product inequalities in \Cref{prop:new-Gcat-prod-ineq} and \Cref{prop:prod-ineq-tck-wrt_G}.
In \Cref{prop:ct-leq-tc-wrtG} and \Cref{prop:ct-leq-strongtc-wrtG}, we prove inequalities 
$$
    \ct_{H}^{\#}(X^{k-1}) \leq \TC_{k,H}^{\#}(X) \leq \TC_{k,G}^{\#}(X) \leq \ct_{G}^{\#}(X^{k})
$$
and 
$$
\ct_{G^{k-1}}^{\#}(X^{k-1}) \leq \TC_{k,G}^{\#,*}(X) \leq \ct_{G^k}^{\#}(X^{k})
$$ 
respectively, where $H$ is a stabilizer subgroup of a point in $X$. 
For a path-connected topological group $A$ with an action of $G$ via topological group homomorphisms, we show that 
$$
    \TC^{\#}_{k,G}(A)=\ct^{\#}_G(A^{k-1})
$$ 
in \Cref{theorem:TC-wrt-G=cat-wrt-G for topological group}.
For any finite group $G$ acting on $S^n$, we compute the exact value of $\ct_{G}^{\#}(S^n)$ in \Cref{example:new-Gcat(S^n)}, and exact values of $\TC^{\#}_{k,G}(S^n)$ and $\TC^{\#,*}_{k,G}(S^n)$ when $n$ is even and both of these invariants can either be $k$ or $k+1$ when $n$ is odd, see \Cref{example:TC(S^n)-wrt-G}. 
Similar computations are also done for $S^1$ acting on $S^n$, see \Cref{example: cat(S^n)-wrt-S^1} and \Cref{example:TC(S^n)-wrt-S^1}.
Various other results are given in \Cref{sec:secat-related-invariants} including the dimension-connectivity upper bound on $\TC_{k,G}^{\#}(X)$ in \Cref{prop:dim-conn-ub}.

In terms of these new invariants we get the following theorem in \Cref{sec:seqtc-fib-bundles}.

\begin{theorem}
\label{thm-for-intro: tck-ub-using-new-inva}
Let $F \hookrightarrow E \xrightarrow{p} B$ be a fibre bundle with structure group $G$ where $E^k$ is a completely normal space. 
Let $\{U_1, \dots, U_m\}$ be an open cover of $B^k$ with sequential motion planners ${s_i \colon U_i \to B^I}$.
If there exists a closed cover $\{R_1, \dots, R_m\}$ of $B$ with local trivializations $f_i \colon p^{-1}(R_i) \to R_i \times F$ of $F \hookrightarrow E \xrightarrow{p} B$ such that $f_i$'s form a $G$-atlas of $F \hookrightarrow E \xrightarrow{p} B$ and $s_i(U_i)\subseteq  R_i^I$, then 
\begin{equation}\label{eq:seqtc-ub-using-gcat}
 \TC_k(E) \leq m+\TC_{k,G}^{\#,*}(F)-1.   
\end{equation}
In particular, $\TC_k(E) \leq m+\ct_{G^k}^{\#}(F^k)-1$ and $\TC_k(E) \leq m+\TC_{k,G}^*(F)-1$.
\end{theorem}

A few attempts have been made in determining how the topological complexity of the total space of a fibre bundle relates to the topological complexities of its fibre and the base space (\cite{Farbergrant}, \cite{daundkarlens}, \cite{Grantfibrations}). 
Later, Paul and Sen generalizing the work of Dranishnikov \cite{strongeqtc} to the sequential setting defined the \emph{sequential strong equivariant topological complexity}, denoted by $\TC_{k,G}^{*}(X)$, and provided the following additive upper bound on $\TC_k(E)$ given by
\begin{equation}\label{eq:Dra}
    \TC_k(E) \leq \TC_k(B)+\TC_{k,G}^{*}(F)-1,
\end{equation}
where $E$, $B$ are locally compact metric ANR-spaces and $G$ is acting properly on $F$, see \cite{PaulSen}. 
In \Cref{remark:TC-wrt-G < TC_G}, we show that $\TC_{k,G}^*(X)$ can be infinite, but $\TC_{k,G}^{\#,*}(X)$ is finite. Hence, the upper bound \eqref{eq:seqtc-ub-using-gcat} could be more appropriate to use for computations than \eqref{eq:Dra} in certain cases, see \Cref{thm:tck-main-example}. 
Furthermore, we state a similar result for $\ct(E)$ in \Cref{sec:ls-category-fibre-bundles} following \cite[Theorem 2.6]{Naskar}.

We recall that the generalized projective product spaces are total spaces of certain fibre bundles, see \Cref{tau-atlas for projective product spaces}.
Then, we provide bounds on the sequential topological complexity of several generalized projective product spaces when the base space is $\R P^1$, see \Cref{prop: cat and tck of gpps with base S1}. 
In \Cref{thm:tck-main-example}, we obtain a lower bound on the sequential topological complexity of generalized projective product spaces when the fibre is the product of various dimensional spheres, see \eqref{eq:seqtc-gpps1}. 
Also, we obtain an upper bound on the sequential topological complexity and compute the LS category for these spaces when the base space is $\mathbb{R}P^1$, see  \eqref{eq:cat-gpps1S1} and \eqref{eq:seqtc-gpps1S1}. 

In \Cref{sec:applica-mapping tori}, we establish an upper bound on the LS-category and the sequential topological complexity of a mapping torus.

\begin{theorem}
    Suppose $f \colon M \to M$ is a homeomorphism of a topological space $M$ such that the corresponding mapping torus $M_f$ is completely normal. 
    If $\{V_1, \dots, V_n\}$ is an open cover of $M$ such that each $V_j$ is invariant under $f$, then $\ct(M_f) \leq n+1$.   
\end{theorem}

\begin{theorem}
    Suppose $f \colon M \to M$ is a homeomorphism of a topological space $M$ and $M_f$ is the corresponding mapping torus. 
    If $(M_f)^k$ is completely normal and $\{V_1, \dots, V_n\}$ is an open cover of $M^k$ such that 
    \begin{enumerate}
        \item[(i)] each $V_j$ admits a continuous section of the free path space fibration $e_{k,M}$, and
        \item[(ii)] each $V_j$ is invariant under $f^k := f \times \dots \times f$,
    \end{enumerate}
    then $\TC_k(M_f) \leq k+n$.
\end{theorem}

\section{Preliminaries}
\label{sec:preliminaries}
In this section, we recall a few definitions and results which help us to compare our notion to the previous ones. 

\begin{definition}[{\cite{Sva,secat}}]
\label{def:secat}
    Suppose $p \colon E \to B$ is a fibration. The \textit{sectional category of $p$}, denoted by $\secat(p)$, is the smallest number $n$ such that there exists an open cover $\{U_i\}^{n}_{i=1}$ of the base $B$ such that over each $U_i$ there exists a continuous section $s_i$ of $p$. If no such $n$ exists, we say $\secat(p) = \infty$.
\end{definition}

If $I$ is an ideal in a commutative ring $R$, then let $\nil(I)$ denote the maximum number of factors in a nonzero product of elements from $I$.

\begin{proposition}[{\cite[Theorem 9.14]{CLOT}}]
    Suppose $p\colon E \to B$ is a fibration. For any commutative ring $R$, we have
    $$
         \nil(\ker p^{*}) + 1 \leq \secat(p),
    $$
    where $p^* \colon H^{*}(B;R) \to H^{*}(E;R)$ is the induced map of cohomology rings.
\end{proposition}

A subset $X_0$ of a topological space $X$ is said to be \emph{categorical} if $X_0$ is contractible in $X$, meaning that the inclusion map $X_0 \hookrightarrow X$ is null-homotopic.
The \emph{Lusternik-Schnirelmann category (LS category)} of $X$, denoted by $\ct(X)$, is the smallest number $n$ such that there exists a categorical open cover $\{U_i\}^{n}_{i=1}$ of $X$. If no such $n$ exists, we say $\ct(X) = \infty$.
Equivalently, for a path-connected space $X$, we have
$$\ct(X) := \secat(e_X),$$
where $e_X\colon P_{x_0}X\to X$ is the path space fibration defined by $e_X(\alpha)=\alpha(1)$.
The cohomological lower bound on the LS category is given in the following proposition.

\begin{proposition}[{\cite[Proposition 1.5]{CLOT}}]
\label{prop: cup leq cat}
    Suppose that $X$ is a topological space. If $R$ is a commutative ring, then
    $$
        \mathrm{cup}_{R}(X) +1  \leq \ct(X),
    $$
    where $\mathrm{cup}_R(X)$ denotes $\nil(\tilde{H}^*(X;R))$.
\end{proposition}


\begin{definition}[\cite{Fadelleqcat}]
\label{def:equi-cat}
    The equivariant \textit{LS category} of a $G$-space $X$, denoted by $\ct_G(X)$, is the smallest number $n$ such that there exists a $G$-invariant open cover $\{U_i\}^{n}_{i=1}$ of $X$ such that each inclusion map $U_i \hookrightarrow X$ is $G$-homotopic to a $G$-map which takes values in a orbit. If no such $n$ exists, we say $\ct_G(X) = \infty$.
\end{definition}

Suppose $X$ is a topological space and $e_{k,X}\colon X^I\to X^k$ is the generalized free path space fibration defined by 
$$
e_{k,X}(\alpha) = \left(\alpha(0), \alpha\left(\frac{1}{k-1}\right), \ldots, \alpha\left(\frac{i}{k-1}\right), \ldots, \alpha\left(\frac{k-2}{k-1}\right), \alpha(1)\right).
$$ 
Then the \emph{$k$-th sequential topological complexity} of $X$, denoted by $\TC_k(X)$, is defined as $$\TC_k(X):=\secat(e_{k,X}).$$



\begin{proposition}[{\cite[Proposition 3.4]{RUD2010}}]
\label{prop: zcl < TC}
    Suppose that $X$ is a path-connected topological space such that $H_i(X)$ is finitely generated for all $i$. 
    Let $\mathbb{K}$ be a field and let $\cup_k$ denote the $k$-fold cup product homomorphism
    $$
        \cup_k \colon H^*(X;\mathbb{K}) \otimes_{\mathbb{K}} \dots \otimes_{\mathbb{K}} H^*(X;\mathbb{K}) \to H^*(X;\mathbb{K}).
    $$
    If $\zl_k(X)$ denotes $\nil(\ker(\cup_k))$, then $\TC_{k}(X) \geq \zl_k(X) + 1$.
\end{proposition}

\section{Sectional category with respect to group actions and induced invariants}
\label{sec:secat-related-invariants} 

For a $G$-space $B$ and a fibration $p \colon E\to B$, we introduce the notion of sectional category with respect to $G$.
For specific fibrations and $G$-actions, we obtain the notions of LS category with respect to $G$, sequential topological complexity with respect to $G$, and strong sequential topological complexity with respect to $G$, which will be used to give upper bounds on sequential topological complexity of the total space of a fibre bundle.
Further, we state the relations between these invariants and well-known invariants: (equivariant) sectional category, LS category, (equivariant) topological complexity, and strong equivariant topological complexity.

\subsection{Sectional category with respect to $G$}

In this subsection, we establish some properties of the sectional category with respect to $G$. 
For example, we establish the fibre-homotopy equivalence and the product inequality.

\begin{definition}
\label{def:secatwrtG}
    Suppose that $B$ is a $G$-space and $p \colon E \to B$ is a fibration. The \textit{sectional category of $p$ with respect to $G$}, denoted by $\secat_G^{\#}(p)$, is the smallest number $n$ such that there exists a $G$-invariant open cover $\{U_i\}^{n}_{i=1}$ of the base $B$ such that over each $U_i$ there exists a continuous section $s_i$ of $p$. If no such $n$ exists, we say $\secat_G^{\#}(p) = \infty$.
\end{definition}

Assume that $E$ and $B$ are $G$-spaces and $p \colon E \to B$ is a $G$-fibration. If the sections $s_i$'s in \Cref{def:secatwrtG} are $G$-maps, then the number $n$ is called the equivariant sectional category of $p$, denoted by $\secat_G(p)$, see \cite[Definition 4.1]{colmangranteqtc}.
We will now state the immediate consequence of \Cref{def:secatwrtG}.

\begin{proposition}
\label{prop:secat-wrtG-properties}
    Suppose that $B$ is a $G$-space and $p \colon E \to B$ is a fibration. Then $p$ has a global section if and only if $\secat_G^{\#}(p) = 1$. In general, $\secat(p) \leq \secat_G^{\#}(p)$, and the equality holds if $G$ acts trivially on $B$. Furthermore,
    \begin{enumerate}
        \item If $E$ is also a $G$-space and $p$ is a $G$-fibration, then $\secat_{G}^{\#}(p) \leq \secat_{G}(p)$.

        \item If $H$ is a subgroup of $G$, then $\secat_{H}^{\#}(p) \leq \secat_{G}^{\#}(p)$. In particular, if $B'$ is another $G$-space and $p'\colon E' \to B'$ is a fibration, then
        $$
        \secat_G^{\#}(p \times p') \leq \secat_{G \times G}^{\#}(p \times p'),
        $$
        where $G$ acts on $B \times B'$ diagonally, and $G \times G$ acts on $B \times B'$ componentwise.
   \end{enumerate}
\end{proposition}

\begin{proposition}
\label{prop:secat-wrtG-ineq}
Consider the following homotopy commutative diagram    
\[
\begin{tikzcd}
E_1 \arrow[r, "\widetilde{f}"] \arrow[d, "p_1"'] & E_2 \arrow[d, "p_2"] \\
B_1 \arrow[r, "f"]                               & B_2\,,                 
\end{tikzcd}
\]
where $p_1$ and $p_2$ are fibrations, and $f$ is a $G$-map between $G$-spaces $B_1$ and $B_2$.
\begin{enumerate}
    \item If $B_1=B_2=B$ and $f$ is the identity map, then $\secat_{G}^{\#}(p_1) \geq \secat_G^{\#}(p_2)$. In particular, the sectional category with respect to a group is a fibre-homotopy equivalence invariant of a fibration.
    
    \item If the diagram is a pullback, then $\secat_{G}^{\#}(p_1) \leq \secat_G^{\#}(p_2)$.
    
    \item If $f$ has a left homotopy inverse $g \colon B_2 \to B_1$ (not necessarily a $G$-map) with a homotopy commutative diagram
\[
\begin{tikzcd}
E_2 \arrow[r, "\widetilde{g}"] \arrow[d, "p_2"'] & E_1 \arrow[d, "p_1"] \\
B_2 \arrow[r, "g"]                               & B_1  ,               
\end{tikzcd}
\]    
then $\secat_{G}^{\#}(p_1) \leq \secat_G^{\#}(p_2)$.
\end{enumerate}
\end{proposition}

\begin{proof}
    The proofs of (1) and (2) are left to the reader. Now we show (3). Suppose that $U$ is a  $G$-invariant open subset of $B_2$ with a section $s_2\colon U \to E_2$ of $p_2$. Let $V=f^{-1}(U)$ and $s_1=\widetilde{g} \circ s_2 \circ f \circ i_V$, where $i_V \colon V\hookrightarrow B_1$ is the inclusion map. Then $V$ is a $G$-invariant open subset of $B_1$ and $s_1$ satisfies 
    $$
    p_1 \circ s_1 
        = p_1 \circ \widetilde{g} \circ s_2 \circ f\ \circ i_V
        \simeq g \circ p_2 \circ s_2 \circ f \circ i_V
        = g \circ f \circ i_V \simeq i_V.
    $$
    As $p_1$ is a fibration and $s_1$ is a homotopy section of $p_1$, we can get a section of $p_1$ over $V$.
\end{proof}

\begin{definition}
    A $G$-space $X$ is said to be $G$-completely normal if for every two $G$-invariant subsets $A$ and $B$ of $X$ satisfying $\overline{A} \cap B = A \cap \overline{B} = \emptyset$, there exist disjoint $G$-invariant open subsets $U$ and $V$ containing $A$ and $B$, respectively.
\end{definition}

\begin{proposition}
\label{prop:new-Gsecat-prod-ineq}
    Suppose that $B_i$ is a $G_i$-space and $p_i \colon E_i \to B_i$ is a fibration for $i=1,2$. If $B_1 \times B_2$ is $(G_1 \times G_2)$-completely normal, then
    $$
        \secat_{G_1\times G_2}^{\#}(p_1\times p_2) \leq \secat_{G_1}^{\#}(p_1) + \secat_{G_2}^{\#}(p_2) - 1,
    $$
    where $G_1 \times G_2$ acts on $B_1 \times B_2$ componentwise.
\end{proposition}

\begin{proof}
Suppose that $\secat_{G_1}^{\#}(p_1)=m$ and $\secat_{G_2}^{\#}(p_2)=n$. Suppose that $\{U_i\}_{i=1}^{m}$ is a $G_1$-invariant open cover of $B_1$ with sections $s_i \colon U_i \to E_1$ of $p_1$ and $\{V_j\}_{j=1}^{n}$ is a $G_2$-invariant open cover of $B_2$ with sections $t_j \colon V_j \to E_2$ of $p_2$. Let $\{X_i\}_{i=0}^{m+n-1}$ be the sequence given by $X_0 = \emptyset$, $X_i =U_1 \cup \dots \cup U_i$ for $1\leq i \leq m$ and $X_i=B_1$ for $m \leq i \leq m+n-1$. Similarly, let $\{Y_j\}_{j=0}^{m+n-1}$ be the sequence given by $Y_0 = \emptyset$, $Y_j = V_1 \cup \dots \cup V_j$ for $1\leq j \leq n$ and $Y_j=B_2$ for $n \leq j \leq m+n-1$.

For each $l \in \{1,2, \dots, m+n-1\}$, define $Z_{i,l}:= (X_i\setminus X_{i-1}) \times (Y_{l-i+1} \setminus Y_{l-i})$ for $1 \leq i \leq m+n-1$. Note that $Z_{i,l}$ are $(G_1 \times G_2)$-invariant and $Z_{i,l}=\emptyset$ for $m+1 \leq i \leq m+n-1$. Define
    $$
        Q_l:=\bigcup_{i=1}^{l} Z_{i,l}
    $$
for $1\leq l \leq m+n-1$. Using similar arguments as in the proof of \cite[Theorem 1.37]{CLOT}, one can show that 
    $
        Z_{i,l} \cap \overline{Z_{i',l}}  = \emptyset
    $
for $i\neq i'$.
Then, by using the property of $(G_1 \times G_2)$-complete normality for the space $B_1 \times B_2$ repeatedly, there exist $(G_1 \times G_2)$-invariant open sets $W_{i,l}$ containing $Z_{i,l}$ such that ${W_{i,l} \cap W_{i',l} = \emptyset}$ for $i \neq i'$.  
Let
$$
C_{i,l} := W_{i,l} \cap (U_i \times V_{l-i+1}).
$$
It is clear that $C_{i,l}$ are $(G_1 \times G_2)$-invariant and $Z_{i,l} \subseteq C_{i,l} \subseteq U_i \times V_{l-i+1}$ and $C_{i,l} \cap C_{i',l} = \emptyset$ for $i \neq i'$. Let 
$$
    C_l:=\coprod_{i=1}^{l} C_{i,l}.
$$
for $1\leq l \leq m+n-1$. Note that $C_l$ are $(G_1 \times G_2)$-invariant and $Q_l \subseteq C_l$ . The composition
\[
\begin{tikzcd}
    C_{i,l} \arrow[r, hook]  &  U_i \times V_{{l-i+1}} \arrow[r,"s_i \times t_{l-i+1}"] & E_1 \times E_2 ,
\end{tikzcd} 
\]
gives a section of $p_1 \times p_2$ on $C_{i,l}$ for each $1 \leq i \leq m$ and $1 \leq l \leq m+n-1$. As $C_l$ is a disjoint union of open sets $C_{i,l}$, we get a section of $p_1 \times p_2$ on $C_l$ for each $1\leq l \leq m+n-1$. Furthermore, one can check that $\bigcup_{l=1}^{m+n-1} Q_l = B_1 \times B_2$.
Hence, $B_1 \times B_2 = \bigcup_{l=1}^{m+n-1} Q_l \subseteq \bigcup_{l=1}^{m+n-1} C_l$ implies $\{C_l\}_{l=1}^{m+n-1}$ is an $(G_1 \times G_2)$-invariant open cover of $B_1 \times B_2$ with sections of $p_1 \times p_2$ on each $C_l$. Hence, $\secat_{G_1\times G_2}^{\#}(p_1 \times p_2) \leq m+n-1$.
\end{proof}

\begin{corollary}
    Suppose that $B$ is a $G$-space and $p_i\colon E_i \to B$ is a fibration for $i=1,2$. Let $E_1 \times_B E_2 = \{(e_1,e_2) \in E_1 \times E_2 \mid p_1(e_1) = p_2(e_2)\}$ and let $p\colon  E_1 \times_B E_2 \to B$ be the fibration given by $p(e_1,e_2)=p_1(e_1)=p_2(e_2)$. If $B \times B$ is $(G \times G)$-completely normal, then
    $$
        \secat_{G}^{\#}(p) \leq \secat_{G}^{\#}(p_1) + \secat_{G}^{\#}(p_2) - 1.
    $$
\end{corollary}

\begin{proof}
Observe that 
\[
\begin{tikzcd}
E_1 \times_B E_2 \arrow[r, hook] \arrow[d, "p"'] & E_1 \times E_2 \arrow[d, "p_1 \times p_2"] \\
B \arrow[r, "\Delta"]                               & B \times B                 
\end{tikzcd}
\]
is a pullback diagram, where $\Delta$ is the diagonal map. Hence,
$$
    \secat_{G}^{\#}(p) 
        \leq \secat_G^{\#}(p_1 \times p_2)
        \leq \secat_{G \times G}^{\#}(p_1 \times p_2) 
        \leq \secat_{G}^{\#}(p_1) + \secat_{G}^{\#}(p_2) - 1
$$
by \Cref{prop:secat-wrtG-ineq} (2), \Cref{prop:secat-wrtG-properties} (3) and \Cref{prop:new-Gsecat-prod-ineq}.
\end{proof}

Although we believe the following result is already known, we could not find a source for it. Therefore, we include it here for the sake of completeness.

\begin{proposition}
\label{prop:eqsecat-prod-ineq}
    Suppose $p_i \colon E_i \to B_i$ is a $G_i$-fibration for $i=1,2$. If $B_1 \times B_2$ is $(G_1 \times G_2)$-completely normal, then
    $$
        \secat_{G_1\times G_2}(p_1\times p_2) \leq \secat_{G_1}(p_1) + \secat_{G_2}(p_2) - 1,
    $$
    where $G_1 \times G_2$ acts on $X \times Y$ componentwise.
\end{proposition}

\begin{proof}
    We follow the same notation as \Cref{prop:new-Gsecat-prod-ineq}, since the proof is similar. If the sections $s_i \colon U_i \to E_1$ of $p_1$ and the sections $t_j \colon V_j \to E_2$ of $p_2$ are $G_1$-equivariant and $G_2$-equivariant, respectively, then the sections of $p_1 \times p_2$ defined using the compositions
    \[
    \begin{tikzcd}
    C_{i,l} \arrow[r, hook]  &  U_i \times V_{{l-i+1}} \arrow[r,"s_i \times t_{l-i+1}"] & E_1 \times E_2 ,
    \end{tikzcd} 
    \]
    are $(G_1 \times G_2)$-equivariant. Thus, the result follows.
\end{proof}

\subsection{LS category with respect to $G$}

\begin{definition}
\label{def:catwrtG}
    Suppose that $X$ is a $G$-space. The \textit{Lusternik-Schnirelmann category (LS category) of $X$ with respect to $G$}, denoted by $\ct_G^{\#}(X)$, is the smallest number $n$ such that there exists a $G$-invariant categorical open cover $\{U_i\}^{n}_{i=1}$ of $X$. If no such $n$ exists, we say $\ct_G^{\#}(X) = \infty$.
\end{definition}


\begin{proposition}
\label{prop:cat-wrtG-properties}
Suppose that $X$ is a $G$-space. Then $X$ is contractible if and only if $\ct_G^{\#}(X) = \ct(X)= 1$. In general,  $\ct(X) \leq \ct_G^{\#}(X)$, and the equality holds if  $G$ acts trivially on $X$. If $H$ is a subgroup of $G$, then $\ct_{H}^{\#}(X) \leq \ct_G^{\#}(X)$. 
In particular, if $Y$ is another $G$-space, then 
    $$
        \ct_G^{\#}(X\times Y) \leq \ct_{G\times G}^{\#}(X\times Y)
    $$ 
where $G$ acts on $X \times Y$ diagonally, and $G \times G$ acts on $X \times Y$ componentwise.
\end{proposition}

The following theorem proves the $G$-homotopy invariance of $\ct_G^{\#}(X)$.

\begin{theorem}
\label{thm:invariance-cat-wrt-G}
    Suppose that $X$ and $Y$ are $G$-spaces. If there exists a $G$-map $g \colon Y \to X$ with a left homotopy inverse, then $\ct_G^{\#}(X) \geq \ct_G^{\#}(Y)$. In particular, the LS category with respect to $G$ of a topological $G$-space is a $G$-homotopy invariant.
\end{theorem}

\begin{proof}
     The proof is similar to \cite[Lemma 1.29]{CLOT}. 
     It just uses the fact that if $U_i$ is $G$-invariant subset of $X$ and $V_i=g^{-1}(U_i)$, then $V_i$ is also $G$-invariant since $g$ is a $G$-map.
\end{proof}

\begin{example}
\label{example:new-Gcat(S^n)}
Let $G$ be a finite group acting on $S^n$, where $n \geq 1$. If $p \in S^n$ and $q \in S^n \setminus G p$, where $G p := \{gp \mid g \in G\}$, then $U_1 = S^n \setminus Gp$ and $U_2 = S^n \setminus Gq$ form a $G$-invariant categorical open cover of $S^n$, i.e., $\ct_{G}^{\#} \left(S^{n}\right) \leq 2$. Hence, $\ct_{G}^{\#} \left(S^{n}\right) = 2$ since $S^{n}$ is not contractible.
\end{example}

\begin{example}\label{example:new-Gcat(S^n)2}
    If $\tau_1$ is the antipodal involution on $S^n$ given by $\tau_1(x)=-x$, then the quotient space $S^n/\left< \tau_1 \right>$ is the real projective space $\mathbb{R}P^n$ with $\ct(\mathbb{R}P^n)=n+1$, see \cite[Example 1.8]{CLOT}. If $\tau_2$ is the reflection involution on $S^n$ given by $\tau_2((x_0,\dots,x_{n-1},x_n))=(x_0,\dots,x_{n-1},-x_n)$, then the quotient space $S^n/\left< \tau_2 \right>$ is the closed disc $D^n$ with $\ct(D^n)=1$, as $D^n$ is contractible. 
\end{example}

Above examples show that $\ct_G^{\#}(X)$ and $\ct(X/G)$ are independent of each other.

\begin{proposition}
\label{prop:catwrtG under transitive action}
    Let $X$ be a $G$-space such that the $G$-action on $X$ is transitive. If $X$ is not contractible, then $\ct_{G}^{\#}(X)=\infty$.
\end{proposition}

\begin{proof}
    As $G$ acts transitively on $X$, the only non-empty $G$-invariant subset of $X$ is $X$. Hence, there does not exist a $G$-invariant categorical open cover of $X$ since $X$ is not contractible.
\end{proof}

\begin{example}
\label{example: cat(G)-wrt-G = infty}
    If a topological group $G$ acts on itself by left multiplication (or right multiplication by inverse), then we have 
    $$
    \ct_{G}^{\#}(G) =
        \begin{cases}
            1 & \text{if $G$ is contractible},\\
            \infty & \text{if $G$ is  not contractible},
        \end{cases}
    $$
    since the action is transitive. On the other hand, $\ct_G(G) = \ct(G/G) = \ct(*) = 1$ if $G$ is metrizable $G$-space, see \cite[Theorem 1.15]{eqlscategory}.
\end{example}

\begin{example}
\label{example: cat(S^n)-wrt-S^1}
    Let $S^1$ act on $S^n \subset  \R^{n+1}$ by matrix multiplication via the embedding
    \[ 
    e^{i\theta} \mapsto
    \begin{pmatrix}
        \cos\theta & -\sin\theta & 0  & \cdots & 0 \\
        \sin\theta & \cos\theta  & 0  & \cdots & 0 \\
        0          & 0           & 1  & \cdots & 0 \\
        \vdots     & \vdots      & \vdots  & \ddots & \vdots \\
        0          & 0           & 0 &  \cdots & 1
    \end{pmatrix} \in SO(n+1).
    \] 
    If $n = 1$, then by \Cref{example: cat(G)-wrt-G = infty}, we have $\ct_{S^1}^{\#}(S^1)=\infty$. If $n \geq 2$, then the coordinate $x_n$ is fixed by the action of $S^1$ on $x = (x_0,\dots,x_n) \in S^n$. 
    Then $U=\{(x_0,\dots,x_n) \in S^n \mid x_n \neq 1\}$ and $V=\{(x_0,\dots,x_n) \in S^n \mid x_n \neq -1\}$ form a $S^1$-invariant categorical open cover of $S^n$. Hence, $\ct_{S^1}^{\#}(S^n)=2$ as $S^n$ is not contractible.
\end{example}

\begin{definition}
    Let $X$ be a $G$-space. A sequence 
    $
    \emptyset = O_0 , O_1 , \dots , O_k = X
    $
    of $G$-invariant open sets is called a categorical sequence with respect to $G$ of length $k$ if for each $1\leq i\leq k$ there exists a $G$-invariant categorical open set $U_i$ such that $O_{i}\setminus O_{i-1}\subseteq U_i$.
\end{definition}

The proofs of the following two propositions is similar to \cite[Lemma 1.36]{CLOT} and \cite[Proposition 9.14]{CLOT}, respectively.

\begin{proposition}\label{prop:cat-seq-wrt-G}
    A $G$-space $X$ has a categorical sequence with respect to $G$ of length $k$ if and only if $\ct_G^{\#}(X)\leq k$.
\end{proposition}

\begin{proposition}
\label{prop:secatwrtG leq catwrtG}
Let $B$ be a $G$-space and $p \colon E\to B$ be a surjective fibration. Then
\begin{itemize}
    \item[(i)] $\secat_{G}^{\#}(p) \leq \ct_{G}^{\#}(B)$.
    \item[(ii)] If $E$ is contractible, then $\secat^{\#}_G(p) = \ct^{\#}_{G}(B)$. 
\end{itemize}
\end{proposition}

Let $X$ be a $G$-space and $x_0 \in X$. 
Recall that the path space fibration $e_X \colon P_{x_0}X \to X$ is given by $e_X(\alpha)=\alpha(1)$. 

\begin{corollary}
\label{cor:new-Gsecat(e_X)=new-Gcat(X)}
If $X$ is a path-connected $G$-space and $x_0 \in X$, then
$
    \secat_G^{\#}(e_X) = \ct_G^{\#}(X).
$
\end{corollary}

\begin{proof}
As $X$ is path-connected, the fibration $e_X$ is surjective. Hence, the claim follows from \Cref{prop:secatwrtG leq catwrtG} (2) since the path space $P_{x_0}X$ is contractible.
\end{proof}

\begin{definition}
A $G$-space $X$ is said to be $G$-connected if 
    $$
        X^H := \{x \in X \mid hx= x \text{ for all }h \in H\}
    $$
is path-connected for every closed subgroup $H$ of $G$.
\end{definition}

\begin{corollary}
\label{cor: cat-wrt-G leq equi-cat}
Suppose $G$ is a Hausdorff topological group. 
If $X$ is a $G$-connected space and $x_0 \in X^G$, then
$
    \ct_G^{\#}(X) \leq \ct_G(X).
$
\end{corollary}

\begin{proof}
    Since $x_0$ is fixed under the $G$-action, it follows that $e_X$ is a $G$-fibration.
    Note that $X$ is path-connected as $X$ is $G$-connected and the trivial subgroup of $G$ is closed.
    Hence, the desired inequality 
    $$
        \ct_G^{\#}(X) = \secat_G^{\#}(e_X) \leq \secat_G(e_X) = \ct_G(X) 
    $$
    follows from \Cref{cor:new-Gsecat(e_X)=new-Gcat(X)}, \Cref{prop:secat-wrtG-properties} and \cite[Corollary 4.7]{colmangranteqtc}.
\end{proof}

\begin{proposition} 
\label{prop:new-Gcat-prod-ineq}
    Suppose that $X$ is a path-connected $G_1$-space and $Y$ is a path-connected $G_2$-space. If $X \times Y$ is $(G_1 \times G_2)$-completely normal, then
    $$
        \ct_{G_1\times G_2}^{\#}(X\times Y) \leq \ct_{G_1}^{\#}(X) + \ct_{G_2}^{\#}(Y) - 1,
    $$
    where $G_1 \times G_2$ acts on $X \times Y$ componentwise.
\end{proposition}

\begin{proof}
Let $(x_0,y_0) \in X \times Y$. 
Then $e_{X \times Y} = e_{X} \times e_{Y}$ under the identification of $P_{(x_0,y_0)}(X\times Y)$ with $P_{x_0} X \times P_{y_0}Y$. Hence, by \Cref{cor:new-Gsecat(e_X)=new-Gcat(X)} and \Cref{prop:new-Gsecat-prod-ineq}, the claim follows.
\end{proof}

\begin{corollary}
\label{cor:new-Gcat-prod-ineq-for-X^k}
    If $G$ is a compact Hausdorff topological group acting continuously on a path-connected Hausdorff topological space $X$ such that $X^k$ is completely normal, then
    $$
    \ct_{G^k}^{\#}(X^k) \leq k(\ct_{G}^{\#}(X)-1)+1.
    $$  
\end{corollary}

\begin{proof}
    A subspace of completely normal space is completely normal. It follows that $X^{k-1}, \dots, X^2$ are completely normal. 
    Hence, by \cite[Lemma 3.12]{colmangranteqtc}, $X^{i}$ is $G^i$-completely normal for $2 \leq i \leq k$. 
    Hence, by the repeated use of \Cref{prop:new-Gcat-prod-ineq}, the result follows.
\end{proof}

\subsection{Sequential topological complexity and strong sequential topological complexity with respect to $G$}

Let $X$ be a $G$-space. Then the free path space $X^{I}$ is a $G$-space with the following action, 
$$ 
    G \times X^I \to X^I,\quad (g \alpha)(t) =  g(\alpha(t)).
$$  
Also, $X^k$ is a $G$-space with respect to the diagonal action and the componentwise $G$-action makes $X^k$ a $G^k$-space. Consider the $k$-points $0 < \frac{1}{k-1} < \cdots < \frac{i}{k-1} < \cdots < \frac{k-2}{k-1} < 1$ on the interval $I$. Define 
$
    e_{k,X} \colon X^I \to X^k
$
by 
$$
e_{k,X}(\alpha) = \left(\alpha(0), \alpha\left(\frac{1}{k-1}\right), \ldots, \alpha\left(\frac{i}{k-1}\right), \ldots, \alpha\left(\frac{k-2}{k-1}\right), \alpha(1)\right).
$$ 
Then $e_{k,X}$ is a $G$-fibration, see \cite[Lemma 3.5]{BaySarkarheqtc}. Next, we introduce two different types of `higher equivariant topological complexities'.

\begin{definition}\label{def:seqtc-wrt-G}
    Suppose that $X$ is a $G$-space. 
    \begin{enumerate}
\item The \textit{$k$-th topological complexity of $X$ with respect to $G$}, denoted by $\TC_{k,G}^{\#}(X)$, is the smallest number $n$ such that there exists a $G$-invariant open cover $\{U_i\}_{i=1}^{n}$ of $X^k$ (where $G$ acts diagonally on $X^k$) such that over each $U_i$ there exists a continuous motion planning $s_i \colon U_i \to X^I$, i.e., $s_i$ is a section of $e_{k,X}$. 
If no such $n$ exists, we say $\TC_{k,G}^{\#}(X) = \infty$. 
In other words, $\TC_{k,G}^{\#}(X) = \secat_{G}^{\#}(e_{k,X})$.

\item The \textit{$k$-th strong topological complexity of $X$ with respect to $G$}, denoted by $\TC_{k,G}^{\#,*}(X)$, is the smallest number $n$ such that there exists a $G^k$-invariant open cover $\{U_i\}_{i=1}^{n}$ of $X^k$ (where $G^k$ acts componentwise on $X^k$) such that over each $U_i$ there exists a continuous motion planning $s_i \colon U_i \to X^I$, i.e., $s_i$ is a section of $e_{k,X}$. 
If no such $n$ exists, we say $\TC_{k,G}^{\#,*}(X) = \infty$. 
In other words, $\TC_{k,G}^{\#,*}(X) = \secat_{G^k}^{\#}(e_{k,X})$.
    \end{enumerate}
\end{definition}

\begin{remark}\label{rem_equitc}
\begin{enumerate}
    \item If the sections $s_i$'s in \Cref{def:seqtc-wrt-G} (1) are $G$-maps, then the number $n$ is called the sequential equivariant topological complexity, denoted by $\TC_{k,G}(X)$, see \cite{BaySarkarheqtc}. If the sections $s_i$'s in \Cref{def:seqtc-wrt-G} (2) are $G$-maps with diagonal action of $G$ on $U_i$'s, then the number $n$ is called the sequential strong equivariant topological complexity, denoted by $\TC_{k,G}^{*}(X)$, see \cite{PaulSen}. We note that the non-sequential versions of these notions were introduced in \cite{colmangranteqtc} and \cite{strongeqtc}.
    \item We note that our notion of $\TC_{k,G}^{\#}(X)$ is the same as the weak equivariant topological complexity $\TC_{k,G}^{w}(X)$ defined by Farber and Oprea \cite{F-O}. 
    Although the motivations for introducing these invariants were different, they are similar in sense that our aim was to find bounds for the topological complexity of the total spaces of fibre bundles, while Farber and Oprea focused on determining bounds for the parametrized topological complexity of fibre bundles.

\end{enumerate}

\end{remark}

We will now state the immediate consequences of \Cref{def:seqtc-wrt-G} and \Cref{rem_equitc}.

\begin{proposition}
\label{prop:new-TC-properties}
Let $X$ be a $G$-space. 
\begin{enumerate}
    \item $X$ is contractible if and only if $\TC_{k,G}^{\#}(X) = \TC_{k,G}^{\#,*}(X) = 1$.
    \item  In general, 
    $$
        \TC_{k,G}^{\#}(X) \leq  \TC_{k,G}(X) \leq \TC_{k,G}^{*}(X), \, \text{and}
    $$
    $$
        \TC_{k}(X) \leq \TC_{k,G}^{\#}(X) \leq \TC_{k,G}^{\#,*}(X) \leq \TC_{k,G}^{*}(X).
    $$
    
    \item If $G$ acts trivially on $X$, then $\TC_{k}(X) = \TC_{k,G}(X) = \TC_{k,G}^{\#}(X) = \TC_{k,G}^{\#,*}(X) = \TC_{k,G}^{*}(X)$.
    
    \item If $H$ is a subgroup of $G$, then $\TC_{k,H}^{\#}(X) \leq \TC_{k,G}^{\#}(X)$ and $\TC_{k,H}^{\#,*}(X) \leq \TC_{k,G}^{\#,*}(X)$. In particular, $Y$ is another $G$-space, then 
    $$
            \TC_{k,G}^{\#}(X \times Y) \leq \TC_{k,G \times G}^{\#}(X \times Y), \,\text{ and }\,  \TC_{k,G}^{\#,*}(X \times Y) \leq \TC_{k,G \times G}^{\#,*}(X \times Y)
    $$ 
    where $G$ acts on $X \times Y$ diagonally, and $G \times G$ acts on $X \times Y$ componentwise.
\end{enumerate}
\end{proposition}
\begin{proof}
Proofs of $(2)$, $(3)$, $(4)$ follow from their respective definitions. $(1)$ can be proved by a similar argument as in \cite[Theorem 1]{FarberTC}. 
\end{proof}

\begin{theorem}\label{thm:invariance-tck-wrt-G}
    Suppose that $X$ and $Y$ are $G$-spaces. If there exists a $G$-map $g \colon Y \to X$ with left homotopy inverse, then $\TC_{k,G}^{\#}(X) \geq \TC_{k,G}^{\#}(Y)$ and $\TC_{k,G}^{\#,*}(X) \geq \TC_{k,G}^{\#,*}(Y)$. Furthermore, the $k$-th topological complexity with respect to $G$ and the $k$-th strong topological complexity with respect to $G$ of a topological $G$-space are $G$-homotopy invariants.
\end{theorem}

\begin{proof}
    The proof is similar to the proof for $\TC_k$, i.e., if $U$ is an open subset of $X^k$ with a section of $e_{k,Y}$ over $U$, then there is a section of $e_{k,X}$ over $V=(g^k)^{-1}(U)$. It just uses the fact that if $U$ is a $G$-invariant (resp. $G^k$-invariant), then $V$ is also $G$-invariant (resp. $G^k$-invariant) since $g^k \colon Y^k \to X^k$ is a $G^k$-map.
\end{proof}

\begin{proposition}\label{prop:prod-ineq-tck-wrt_G}
Suppose that $X$ is a $G_1$-space and $Y$ is a $G_2$ space. If the product $X^k \times Y^k$ is a $(G_1 \times G_2)$-completely normal space, then
$$
   \TC_{k,G_1\times G_2}^{\#}(X\times Y) \leq \TC_{k,G_1}^{\#}(X) + \TC_{k,G_2}^{\#}(Y) - 1,
$$
If $X^k \times Y^k$ is a $(G_1^k \times G_2^k)$-completely normal, then
$$
    \TC_{k,G_1\times G_2}^{\#,*}(X\times Y) \leq \TC_{k,G_1}^{\#,*}(X) + \TC_{k,G_2}^{\#,*}(Y) - 1,
$$
where $G_1 \times G_2$ acts on $X \times Y$ componentwise.
\end{proposition}

\begin{proof}
Under the natural identifications of $(X\times Y)^I = X^I \times Y^I$ and $(X \times Y)^k = X^k \times Y^k$, we have $e_{k,X\times Y} = e_{k,X} \times e_{k,Y}$. Then, by \Cref{prop:new-Gsecat-prod-ineq}, the result follows.
\end{proof}


    


\begin{proposition}\label{prop: fixed-set-inequality}
    Suppose that $G$ is a compact Hausdorff topological group and $X$ is a Hausdorff $G$-space. If $H$ and $K$ are closed subgroups of $G$ such that $X^H$ is $K$-invariant, then 
    \[
    \TC_{k,K}^{\#}(X^H)\leq \TC_{k,G}(X) \quad \text{and} \quad \TC_{k,K}^{\#,*}(X^H)\leq \TC_{k,G}^{*}(X).
    \]
\end{proposition}
\begin{proof}
It follows from \cite[Proposition 3.14]{BaySarkarheqtc} that $\TC_{k,K}(X^H)\leq \TC_{k,G}(X)$. 
Similarly, for strong equivariant topological complexity we have $\TC_{k,K}^{*}(X^H)\leq \TC_{k,G}^{*}(X)$, see \cite[Proposition 5.6]{PaulSen}. Then we get the desired inequalities from \Cref{prop:new-TC-properties} (2).
\end{proof}



\begin{proposition}
\label{prop:ct-leq-tc-wrtG}
Let $X$ be a path-connected $G$-space and $H$ be a stabilizer of some $x_0\in X$. Then
\[
\ct_{H}^{\#}(X^{k-1}) \leq \TC_{k,H}^{\#}(X) \leq\TC_{k,G}^{\#}(X) \leq \ct_{G}^{\#}(X^k).
\]
\end{proposition}
\begin{proof}
The middle inequality follows from \Cref{prop:new-TC-properties} (5), and the right inequality follows from \Cref{prop:secatwrtG leq catwrtG} (1).
We will now prove the left inequality.
Define a map $f\colon X^{k-1}\to X^k$ by $f(x_1,\dots,x_{k-1})=(x_0,x_1,\dots,x_{k-1})$.
Then note that $f$ is an $H$-equivariant map. 
Now consider the following pullback diagram
\[ 
    \begin{tikzcd}
Q \arrow{r}{} \arrow[swap]{d}{q} & X^I \arrow{d}{e_{k,X}} \\
X^{k-1} \arrow{r}{f}&X^k
    \end{tikzcd},
\] 
where the space $Q := \{\gamma\in X^I \mid \gamma(0)=x_0\}$ and $q \colon Q \to X^{k-1}$ is the map given by ${q(\gamma) =  \left( \gamma\left(\frac{1}{k-1}\right), \ldots, \gamma\left(\frac{i}{k-1}\right), \ldots, \gamma\left(\frac{k-2}{k-1}\right), \gamma(1)\right)}$.
Then, by \Cref{prop:secat-wrtG-ineq} (2), we have $\secat_H^{\#}(q)\leq \secat_{H}^{\#}(e_{k,X}) = \TC_{k,H}^{\#}(X)$. 

Let $\bar{x}_0=(x_0,\dots,x_0) \in X^{k-1}$. Define a map $\phi\colon Q \to P_{\bar{x}_0}X^{k-1}$ by $\phi(\gamma)=(\gamma_1,\dots,\gamma_{k-1})$, where $\gamma_i \colon [0,1] \to X$ is the path $\gamma$ restricted to the interval $\left[0,\frac{i}{k-1}\right]$. Note that $\gamma_i(0)=x_0$ and $\gamma_i(1)=\gamma\left(\frac{i}{k-1}\right)$. Hence, the diagram
\[
\begin{tikzcd}
Q \arrow[rr, "\phi"] \arrow[rd, "q"'] &     & P_{\bar{x}_0}X^{k-1} \arrow[ld, "{e_{X^{k-1}}}"] \\
                                      & X^{k-1} &                              
\end{tikzcd}
\]
commutes. Thus, by \Cref{cor:new-Gsecat(e_X)=new-Gcat(X)} and \Cref{prop:secat-wrtG-ineq} (1), it follows that $\ct_{H}^{\#}(X^{k-1}) = \secat_{H}^{\#}(e_{X^{k-1}}) \leq \secat_{H}^{\#}(q)$. Hence, we get $\ct_{H}^{\#}(X^{k-1}) \leq \secat_{H}^{\#}(q) \leq \TC_{k,H}^{\#}(X)$.
\end{proof}



It is clear that if $X$ is a path-connected $G$-space with $X^G \neq \emptyset$, then, by \Cref{prop:ct-leq-tc-wrtG}, it follows that
$$
        \TC_{k,G}^{\#}(X) \leq \ct_{G}^{\#}(X^k) \leq \TC_{k+1,G}^{\#}(X),
$$
since the stabilizer subgroup of a point in $X^G$ is $G$. As Prof. Oprea pointed out, the fixed-point condition can removed. We  now provide the proof suggested by him.

\begin{proposition}
\label{prop: TC_k-wrt-G leq TC_(k+1)-wrt-G}
    For a $G$-space $X$, we have  
    $\TC_{k,G}^{\#}(X) \leq \TC_{k+1,G}^{\#}(X)$ for all $k \geq 2$.
\end{proposition}

\begin{proof}
    Suppose $\mu \colon X^k \to X^{k+1}$ is the map defined by $\mu(x_1,\dots,x_k) = (x_1,\dots ,x_k,x_k)$. Note that $X^k$ is a retract of $X^{k+1}$, where $X^k$ sits inside $X^{k+1}$ via $\mu$, with retraction given by the projection map onto the first $k$-coordinates. Hence, for any topological space $Y$, the induced map on homotopy classes of maps $\mu_* \colon [Y, X^k] \to [Y, X^{k+1}]$ is injective.
    
    Now we show that $\mu \circ e_{k,X} \simeq e_{k+1,X}$, i.e., the following diagram
    \[
    \begin{tikzcd}
               & X^I \arrow[ld, "{e_{k,X}}"'] \arrow[rd, "{e_{k+1,X}}"] &         \\
    X^k \arrow[rr, "\mu"] &                        & X^{k+1}
    \end{tikzcd}
    \]
    commutes up to homotopy. Define a homotopy $H \colon X^I \times I \to X^{k+1}$ by
    $$
        H(\gamma,t) 
            := (\gamma(0), \gamma(s_1(t)), \dots,\gamma(s_i(t)), \dots, \gamma(s_{k-1}(t)),\gamma(1)),
    $$
    where $s_i$ is the line segment joining the point $i/(k-1)$ to $i/k$ given by
    $$
        s_i(t) = \frac{(1-t)i}{k-1}+\frac{ti}{k}
    $$
    for $1 \leq i \leq k-1$. Then
    $$
        H(\gamma,0) 
            = \left(\gamma(0), \gamma\left(\frac{1}{k-1}\right), \dots,\gamma\left(\frac{i}{k-1}\right), \dots, \gamma\left(1\right),\gamma(1)\right)
            = (\mu \circ e_{k,X})(\gamma)
    $$
    and
    $$
        H(\gamma,1) 
            = \left(\gamma(0), \gamma\left(\frac{1}{k}\right), \dots,\gamma\left(\frac{i}{k}\right), \dots, \gamma\left(\frac{k-1}{k}\right),\gamma(1)\right)
            = e_{k+1,X}(\gamma).
    $$
    Suppose $U$ is a $G$-invariant open subset of $X^{k+1}$ with a section $s \colon U \to X^I$ of $e_{k+1,X}$. Then $V = \mu^{-1}(U)$ is a $G$-invariant open subset of $X^k$. Let $t = s \circ \left.\mu\right|_{V} \colon V \to X^I$. Then we have
    $$
     \mu \circ e_{k,X} \circ t 
        \simeq e_{k+1,X} \circ t
        = e_{k+1,X} \circ s \circ \left.\mu\right|_{V}
        = \left.\mu\right|_{V}.
    $$
    As $\mu_*$ is injective, we get $e_{k,X} \circ t \simeq i_{V} \colon V \hookrightarrow X^{k}$, i.e, $t$ is a homotopy section of $e_{k,X}$ over $V$. As $e_{k,X}$ is a fibration, we can get a section of $e_{k,X}$ over $V$.
\end{proof}


\begin{proposition}
\label{prop:ct-leq-strongtc-wrtG}
Let $X$ be a path-connected $G$-space. Then
\[   
    \ct_{G^{k-1}}^{\#}(X^{k-1})\leq\TC_{k,G}^{\#,*}(X) \leq \ct_{G^{k}}^{\#}(X^k).
\]
\end{proposition}

\begin{proof}
The right inequality follows from \Cref{prop:secatwrtG leq catwrtG} (1). 
We will now prove the left inequality. Let $x_0 \in X$ and $f\colon X^{k-1}\to X^k$ be a map defined as $f(x_1,\dots,x_{k-1})=(x_0,x_1,\dots,x_{k-1})$. Then note that $f$ is a $(G^{k-1})$-equivariant map where $G^{k-1}$ acts on $X^{k}$ as a subgroup $\{e\} \times G^{k-1}$ of $G^k$ (here $e \in G$ is the identity element). 
Considering the same pullback diagram as in \Cref{prop:ct-leq-tc-wrtG}, by \Cref{prop:secat-wrtG-ineq} (2) and \Cref{prop:secat-wrtG-properties}, we have $\secat_{G^{k-1}}^{\#}(q)\leq \secat_{G^{k-1}}^{\#}(e_{k,X}) \leq \secat_{G^{k}}^{\#}(e_{k,X}) = \TC_{k,G}^{\#,*}(X)$.  
Using the same argument as in \Cref{prop:ct-leq-tc-wrtG}, we have $\ct_{G^{k-1}}^{\#}(X^{k-1}) = \secat_{G^{k-1}}^{\#}(e_{X^{k-1}}) \leq \secat_{G^{k-1}}^{\#}(q)$. 
Hence, we get the left inequality $\ct_{G^{k-1}}^{\#}(X^{k-1}) \leq \secat_{G^{k-1}}^{\#}(q) \leq \TC_{k,G}^{\#,*}(X)$.
\end{proof}

\begin{corollary}
    Suppose that $X$ is a path-connected $G$-space, then $\TC_{k,G}^{\#,*}(X) \leq \TC_{k+1,G}^{\#,*}(X)$ for all $k \geq 2$.
\end{corollary}

\begin{proof}
    By \Cref{prop:ct-leq-strongtc-wrtG}, it follows that $\TC_{k,G}^{\#,*}(X) \leq \ct_{G^{k}}^{\#}(X^k) \leq \TC_{k+1,G}^{\#,*}(X)$.
\end{proof}

\begin{corollary}
\label{cor:TC-wrt-G under transitive action}
    Let $X$ be a path-connected $G$-space such that the $G$-action on $X$ is transitive. If $X$ is not contractible, then $\TC_{k,G}^{\#,*}(X)=\infty$. Moreover, if $X$ has a fixed point, then $\TC_{2,G}^{\#}(X)=\infty$.
\end{corollary}
\begin{proof}
    By \Cref{prop:catwrtG under transitive action} and \Cref{prop:ct-leq-strongtc-wrtG}, we have $\infty = \ct_{G^{k-1}}^{\#}(X^{k-1}) \leq \TC_{k,G}^{\#,*}(X)$, as the action of $G^{k-1}$ on $X^{k-1}$ is transitive. 
    If $x_0 \in X$ is a fixed point of $X$ under $G$-action, then the stabilizer subgroup of $x_0$ is $G$. Hence, by \Cref{prop:catwrtG under transitive action} and \Cref{prop:ct-leq-tc-wrtG}, we get $\infty = \ct_{G}^{\#}(X) \leq \TC_{2,G}^{\#}(X)$.
\end{proof}

\begin{example}
\label{example: TC(G)-wrt-G}
    If a topological group $G$ acts on itself by left multiplication (or right multiplication by inverse), then the action is transitive and we have 
    $$
    \TC_{k,G}^{\#,*}(G) =
        \begin{cases}
            1 & \text{if $G$ is contractible},\\
            \infty & \text{if $G$ is  not contractible}.
        \end{cases}
    $$
\end{example}


\begin{theorem}
\label{theorem:TC-wrt-G=cat-wrt-G for topological group}
   Let $A$ be a path-connected topological group with an action of $G$ via topological group homomorphisms. 
   Then $TC^{\#}_{k,G}(A)=\ct^{\#}_G(A^{k-1})$.
\end{theorem}

\begin{proof}
Since $G$ acts on $A$ via topological group homomorphisms and on $A^{k-1}$ diagonally, the identity element $\bar{e}=(e,\dots,e)\in A^{k-1}$ becomes a fixed point of $G$-action on $A^{k-1}$.
Thus, the inequality $\ct_{G}^{\#}(A^{k-1})\leq\TC_{k,G}^{\#}(A)$ follows from \Cref{prop:ct-leq-tc-wrtG}. 

We now prove the other inequality. 
Let $f\colon A^k \to A^{k-1}$ be the map given by $f(a_1,\dots,a_k)=(a_2a_1^{-1},a_3a_2^{-1},\dots,a_ka_{k-1}^{-1})$. Note that $f$ is a $G$-map since
\begin{align*}
    f(g (a_1,\dots,a_k)) & = f(ga_1,\dots,ga_k) \\
        & = \left((ga_2)(ga_1)^{-1},(ga_3)(ga_2)^{-1},\dots,(ga_k)(ga_{k-1})^{-1}\right) \\
        & = \left((ga_2)(ga_1^{-1}),(ga_3)(ga_2^{-1}),\dots,(ga_k)(ga_{k-1}^{-1})\right) \\
        & = \left(g(a_2a_1^{-1}),g(a_3a_2^{-1}),\dots,g(a_ka_{k-1}^{-1})\right) \\
        & = g(a_2a_1^{-1},a_3a_2^{-1},\dots,a_ka_{k-1}^{-1})= gf(a_1,\dots,a_k).
\end{align*}
Now consider the pullback diagram
    \[
    \begin{tikzcd}
    Q \arrow[r, "\widetilde{f}"] \arrow[d, "q"'] & P_{\bar{e}}A^{k-1} \arrow[d, "e_{A^{k-1}}"] \\
    A^k \arrow[r, "f"]                               & A^{k-1}                 
    \end{tikzcd},
    \]
where $Q = \{(a_1,\dots,a_k,\gamma) \in A^k \times P_{\bar{e}}A^{k-1} \mid \gamma(1) = (a_2a_1^{-1},a_3a_2^{-1},\dots,a_ka_{k-1}^{-1})\}$ and $q\colon Q \to A^k$ is the projection map. 
Then, by \Cref{prop:secat-wrtG-ineq} (2) and \Cref{cor:new-Gsecat(e_X)=new-Gcat(X)}, it follows $\secat_{G}^{\#}(q) \leq \secat_G^{\#}(e_{A^{k-1}}) = \ct_{G}^{\#}(A^{k-1})$.

Define a map $\phi\colon Q \to A^{I}$ by $\phi(a_1,\dots,a_k, \gamma) = \gamma_1a_1*\gamma_2a_2*\cdots * \gamma_{k-1}a_{k-1}$, where $\gamma = (\gamma_1,\dots,\gamma_{k-1})$, the path $\gamma_ia_i$ is defined as $(\gamma_i a_i)(t):=\gamma_i(t)a_i$ and where $*$ denotes the concatenation of paths. Note that $(\gamma_ia_i)(0)=\gamma_i(0)a_i=ea_i=a_i$ and $(\gamma_ia_i)(1)=\gamma_i(1)a_i=(a_{i+1}a_i^{-1})a_i=a_{i+1}$. Hence, the diagram 
\[
\begin{tikzcd}
Q \arrow[rr, "\phi"] \arrow[rd, "q"'] &     & A^{I} \arrow[ld, "{e_{k,A}}"] \\
                                      & A^k &                              
\end{tikzcd}
\]
commutes. Thus, by \Cref{prop:secat-wrtG-properties} (1), it follows $\TC_{k,G}^{\#}(A) = \secat_{G}^{\#}(e_{k,A}) \leq \secat_{G}^{\#}(q)$. Hence, we get the right inequality $\TC_{k,G}^{\#}(A) \leq \secat_{G}^{\#}(q) \leq \ct_{G}^{\#}(A^{k-1})$.
\end{proof}

\begin{example}
\label{example:TC(S^n)-wrt-G}
    Suppose that $G$ is a finite group acting on $S^n$, where $n \geq 1$. Using \Cref{prop:new-TC-properties}, \Cref{prop:ct-leq-strongtc-wrtG}, \Cref{cor:new-Gcat-prod-ineq-for-X^k}, and \Cref{example:new-Gcat(S^n)}, we get
    $$
        \TC_{k,G}^{\#}(S^n) \leq \TC_{k,G}^{\#,*}(S^n) \leq \ct_{G^k}^{\#}((S^n)^k) \leq k(\ct_{G}^{\#}(S^n)-1)+1 = k+1.
    $$
    Moreover, by \Cref{prop:new-TC-properties}, we also have the inequality $\TC_k(S^n) \leq \TC_{k,G}^{\#}(S^n)$. As $\TC_{k}(S^n)=k$ for $n$ odd and $\TC_k(S^n)=k+1$ for $n$ even, see \cite[Section 4]{RUD2010}, we get 
    $$
        \TC_{k,G}^{\#}(S^n) = \TC_{k,G}^{\#,*}(S^n) = k+1
    $$
    for $n$ even, and
    $$
       k \leq \TC_{k,G}^{\#}(S^n) \leq \TC_{k,G}^{\#,*}(S^n) \leq k+1
    $$
    for $n$ odd.

    If $\tau$ is the antipodal involution on $S^n$ given by $\tau(z) = -z$, then $\TC_{2,\left< \tau \right>}^{\#}(S^n)=2$ for $n$ odd. This follows because the sets $U_1 = \{(x,y) \in S^n \times S^n \mid x \neq -y\}$ and $U_2 = \{(x,y) \in S^n \times S^n \mid x \neq y \}$ are $\left< \tau \right>$-invariant open and admit continuous motion planners, see \cite[Theorem 8]{FarberTC}.
    
    If $\sigma$ is the conjugation involution on $S^1 \subset \mathbb{C}$ given by $\sigma(z)=\bar{z}$, then $\left<\sigma \right>$ acts on $S^1$ via group homomorphisms. Hence, by \Cref{theorem:TC-wrt-G=cat-wrt-G for topological group}, \Cref{prop:cat-wrtG-properties}, \Cref{cor:new-Gcat-prod-ineq-for-X^k}, and \Cref{example:new-Gcat(S^n)}, we have
    $$
        \TC^{\#}_{k,\left<\sigma\right>}(S^1)
            = \ct^{\#}_{\left<\sigma\right>}((S^1)^{k-1}) 
            \leq \ct^{\#}_{\left< \sigma \right>^{k-1}}((S^1)^{k-1}) 
            \leq (k-1)(\ct_{\left< \sigma \right>}^{\#}(S^1)-1)+1 
            = k.
    $$
    As $k = \TC_{k}(S^1) \leq \TC_{k,\left< \sigma \right>}^{\#}(S^1)$, we get $\TC_{k,\left< \sigma \right>}^{\#}(S^1) = k$.
    
Similarly, if $\eta$ is the involution on $S^3 \subset \mathbb{R}^4$ given by $\eta(x_0,x_1,x_2,x_3)=(x_0,-x_1,x_2,-x_3)$, then $\left<\eta \right>$ acts on $S^3$ via group homomorphisms. Hence, similar computations like above show that $\TC_{k,\left< \eta \right>}^{\#}(S^3) = k$.
\end{example}

\begin{remark}
\label{remark:TC-wrt-G < TC_G}
    The previous example shows that the inequalities $\TC_{k,G}^{\#}(X) \leq  \TC_{k,G}(X)$ and $\TC_{k,G}^{\#,*}(X) \leq  \TC_{k,G}^{*}(X)$ can be strict. Since $\TC_{k,\left< \sigma \right>}^{\#}(S^1) = k$ and $\TC_{k,\left< \sigma \right>}^{\#,*}(S^1) \leq k+1$ but $\TC_{k,\left<\sigma\right>}(S^1)$ and $\TC_{k,\left<\sigma\right>}^{*}(S^1)$ are infinite as $S^1$ is not $\left< \sigma \right>$-connected, see \cite[Proposition 3.14 (2)]{BaySarkarheqtc} and \cite[Proposition 5.6 (c)]{PaulSen}.
\end{remark}

\begin{example}
\label{example:TC(S^n)-wrt-S^1}
    Let $S^1 \subseteq SO(n+1)$ act on $S^n \subset \R^{n+1}$ by matrix multiplication as in \Cref{example: cat(S^n)-wrt-S^1}. 
    If $n = 1$, then by \Cref{example: TC(G)-wrt-G}, we have $\TC_{k,S^1}^{\#,*}(S^1)=\infty$. 
    Note that $\TC_{2,S^1}^{\#}(S^1)=2$ since $2= \TC_{2}(S^1) \leq \TC_{2,S^1}^{\#}(S^1) \leq \TC_{2,S^1}(S^1)=2$, see \cite[Theorem 8]{FarberTC} and \cite[Example 5.10]{colmangranteqtc}.
    
    If $n \geq 2$, then by \Cref{example: cat(S^n)-wrt-S^1}, we have $\ct_{S^1}^{\#}(S^n)=2$. Following similar computations as \Cref{example:TC(S^n)-wrt-G}, we get 
    $$
        \TC_{k,S^1}^{\#}(S^n) = \TC_{k,S^1}^{\#,*}(S^n) = k+1 
    $$
    for $n$ even, and
    $$
       k \leq \TC_{k,S^1}^{\#}(S^n) \leq \TC_{k,S^1}^{\#,*}(S^n) \leq k+1
    $$
    for $n$ odd and $n \neq 1$.
\end{example}

In the following proposition, we obtain the dimension-connectivity upper bound for the sequential equivariant topological complexity.

\begin{proposition}\label{prop:dim-conn-ub}
Let $X$ be a
$G$-CW-complex of dimension at least 1 such that $X^H$ is $m$-connected for all subgroups $H \leq G$. Then 
 \[\TC_{k,G}^{\#}(X) \leq \TC_{k,G}(X) < \frac{k \dim(X)+1}{m+1} + 1.
 \]     
\end{proposition}
\begin{proof}
 Since $e_{k,X}$ is a $G$-fibration, it is also a Serre $G$-fibration. 
 For a closed subgroup $H$ of $G$, we have $(X^I)^H=(X^H)^I$ and $(X^k)^H=(X^H)^k$. Thus, the fibration $e_{k,X}^H\colon (X^I)^H\to (X^k)^H$ coincides with the fibration $e_{k,X^H}\colon (X^H)^I\to (X^H)^k$.
 Note that the fibre of   $e_{k,X^H}$ is $(\Omega X^H)^{k-1}$ which is $(m-1)$-connected.  Therefore, from \cite[Theorem 3.5]{grant2019symmetrized}, we obtain the right inequality.
 The left inequality follows from \Cref{prop:new-TC-properties} (2).
\end{proof}

    




\section{Sequential topological complexity of fibre bundles}
\label{sec:seqtc-fib-bundles}

Let $F \hookrightarrow E \xrightarrow{p} B$ be a fibre bundle. In \cite[Theorem 2.2]{DaundSarkargpps}, the second and third author gave an additive upper bound on the topological complexity of $E$. 
In this section, we generalize that result in \Cref{additive inequality for TC for fibre bundles} to the sequential setting and obtain other additive upper bounds in terms of invariants introduced in the previous section. 

\begin{definition}
\label{def: induced trivialization for TC}
    Let $F \hookrightarrow E \xrightarrow{p} B$ be a fibre bundle and $U$ be a subset of $B^k$. Suppose that $s \colon U \to B^I$ is a sequential motion planner over $U$ such that $s(U)\subseteq  R^I$ where $f \colon p^{-1}(R) \to R \times F$ is a local trivialization of $F \hookrightarrow E \xrightarrow{p} B$. Then $f^k \colon (p^k)^{-1}(R^k) \to R^k \times F^k$ gives a local trivialization of $F^k \hookrightarrow E^k \xrightarrow{p^k} B^k$ over $R^k$. As $s(U) \subseteq  R^I$, it follows that $U \subseteq R^k$. Hence, we obtain a local trivialization
    $$
    h: = \left.f^k\right|_{(p^{k})^{-1}(U)} \colon (p^k)^{-1}(U) \to U \times F^k
    $$
of the fibre bundle $F^k \hookrightarrow E^k \xrightarrow{p^k} B^k$ over $U$. We will call this local trivialization $h$ the induced trivialization of the fibre bundle $F^k \hookrightarrow E^k \xrightarrow{p^k} B^k$ with respect to the sequential motion planner $s$ and the trivialization $f$.
\end{definition}

\begin{theorem}
\label{additive inequality for TC for fibre bundles}
    Let $F \hookrightarrow E \xrightarrow{p} B$ be a fibre bundle where $E^k$ is a completely normal space. Let $\{U_1, \dots, U_m\}$ and  $\{V_1,\dots,V_n\}$ be open covers of $B^k$ and $F^k$ with sequential motion planners $s_i \colon U_i \to B^I$ and $t_j \colon V_j \to F^I$ respectively. If
    \begin{itemize}
        \item[(i)] there exists a cover $\{R_1, \dots, R_m\}$ of $B$ with local trivializations $f_i \colon p^{-1}(R_i) \to R_i \times F$ of $F \hookrightarrow E \xrightarrow{p} B$  such that $s_i(U_i)\subseteq  R_i^I$, and
        \item[(ii)] the induced local trivializations $h_i \colon (p^{k})^{-1}(U_i) \to U_i \times F^k$ with respect to $s_i$ and $R_i$ extend to local trivializations over $\overline{U_i}$ such that $(\overline{U_i} \cap \overline{U_{i'}}) \times V_j$ is invariant under $h_{i'} \circ h_i^{-1}$ for all $1 \leq i , i' \leq m$ and $1 \leq j \leq n$,
    \end{itemize}
then $\TC_k(E) \leq m+ n-1$.
\end{theorem}

\begin{proof}
Let $\{X_i\}_{i=0}^{m+n-1}$ be the sequence given by $X_0 = \emptyset$, $X_i =U_1 \cup \dots \cup U_i$ for $1\leq i \leq m$ and $X_i=B^k$ for $m \leq i \leq m+n-1$. Similarly, let $\{Y_j\}_{j=0}^{m+n-1}$ be the sequence given by $Y_0 = \emptyset$, $Y_j = V_1 \cup \dots \cup V_j$ for $1\leq j \leq n$ and $Y_j=F^k$ for $n \leq j \leq m+n-1$.

For each $l \in \{1,2, \dots, m+n-1\}$, define $Z_{i,l}:= (X_i\setminus X_{i-1}) \times (Y_{l-i+1} \setminus Y_{l-i})$ for $1 \leq i \leq m+n-1$. Note that $Z_{i,l}=\emptyset$ for $m+1 \leq i \leq m+n-1$. Define
    \begin{equation}
    \label{W_l for TC}
        Q_l:=\bigcup_{i=1}^{l} h_i^{-1}(Z_{i,l})
    \end{equation}
for $1\leq l \leq m+n-1$. Note that $h_i$'s are defined from $1 \leq i \leq m$, hence the term $h_i^{-1}(Z_{i,l})$ in \eqref{W_l for TC} are assumed to be the empty set for each $m+1 \leq i \leq m+n-1$ as $Z_{i,l}=\emptyset$. We claim that the terms in the right hand side of \eqref{W_l for TC} are separated, i.e.,
    $$
        h_i^{-1}(Z_{i,l}) \cap \overline{h_{i'}^{-1}(Z_{i',l})} = \emptyset
    $$
for $i\neq i'$, where the closure $\overline{h_{i'}^{-1}(Z_{i',l})}$ is taken in $E^k$. 
Suppose that $z \in h_i^{-1}(Z_{i,l}) \cap \overline{h_{i'}^{-1}(Z_{i',l})}$. 
Note that $Z_{i',l} \subseteq U_{i'} \times V_{l-i'+1} \subseteq U_{i'} \times F^k$ implies $\overline{Z_{i',l}} \subseteq \overline{U_{i'} \times F^k} \subseteq \overline{U_{i'}} \times \overline{F^k} = \overline{U_{i'}} \times F^k$. 
Since $h_{i'}$ extends to a local trivialization over $\overline{U_{i'}}$, we have $h_{i'}^{-1}(\overline{Z_{i',l}})$ is a closed subset of $(p^k)^{-1}(\overline{U_{i'}})$, which is closed in $E^k$.
Hence, $h_{i'}^{-1}(\overline{Z_{i',l}})$ is a closed subset in $E^k$ which contains $h_{i'}^{-1}(Z_{i',l})$.
This implies $\overline{h_{i'}^{-1}(Z_{i',l})} \subseteq h_{i'}^{-1}(\overline{Z_{i',l}})$. 
Hence,
\begin{align*}
    h_{i'}(z)   \in h_{i'}(\overline{h_{i'}^{-1}(Z_{i',l})}) 
                \subseteq h_{i'}(h_{i'}^{-1}(\overline{Z_{i',l}}))
                & = \overline{Z_{i',l}} \\
                & = \overline{(X_{i'}\setminus X_{i'-1}) \times (Y_{l-i'+1} \setminus Y_{l-i'})}\\
                & \subseteq \overline{(X_{i'}\setminus X_{i'-1})} \times \overline{(Y_{l-i'+1} \setminus Y_{l-i'})}.
\end{align*} 

Suppose that $h_{i'}(z)=(x_{i'},y_{i'})$, where $x_{i'}=p^k(z)=x$ (say). 
As $X_{i'-1}$ is open, the following inclusions
$$
   x \in \overline{(X_{i'}\setminus X_{i'-1})} 
        = \overline{(X_{i'} \cap X_{i'-1}^c)}
        \subseteq \overline{X_{i'}} \cap \overline{X_{i'-1}^c}
        = \overline{X_{i'}} \cap X_{i'-1}^c
$$
imply $x \notin X_{i'-1}$.
Similarly, $y_{i'} \not\in {Y_{l-i'}}$. 

Suppose that $h_{i}(z)=(x_{i},y_{i})$, where $x_{i}=p^k(z)=x$. Then 
$$
(x,y_{i})=(x_{i},y_{i})=h_{i}(z) \in h_{i}(h_i^{-1}(Z_{i,l})) = Z_{i,l} \subseteq U_i \times V_{l-i+1}
$$
and $(x,y_{i'})=(x_{i'},y_{i'})=h_{i'}(z) \in \overline{(X_{i'}\setminus X_{i'-1})} \times \overline{(Y_{l-i'+1} \setminus Y_{l-i'})} \subseteq \overline{U_{i'}} \times \overline{(Y_{l-i'+1} \setminus Y_{l-i'})}$ implies $(x,y_i) \in (\overline{U_i} \cap \overline{U_{i'}}) \times V_{l-i+1}$. Hence $(x,y_{i'}) = h_{i'}(h_{i}^{-1}(x,y_i)) \in (\overline{U_i} \cap \overline{U_{i'}}) \times V_{l-i+1}$ as $(\overline{U_i} \cap \overline{U_{i'}}) \times V_{l-i+1}$ is preserved under $h_{i'}\circ h_{i}^{-1}$. In particular, we have $y_{i'} \in V_{l-i+1}$. 

If $i<i'$,  then $i \leq i'-1$ implies $X_{i} \subseteq X_{i'-1}$. This is a contradiction since $x\in X_i \setminus X_{i-1} \subseteq X_i$ but $x\not\in X_{i'-1}$. If $i'< i$, then $l-i' \geq l-i+1$ implies $Y_{l-i+1} \subseteq Y_{l-i'}$. This is a contradiction since $y_{i'} \in V_{l-i+1} \subseteq Y_{l-i+1}$ but $y_{i'} \not\in Y_{l-i'}$.
    
Then, by using the property of complete normality for the space $E^k$ repeatedly, there exist open sets $W_{i,l}$ containing $h_i^{-1}(Z_{i,l})$ such that $W_{i,l} \cap W_{i',l} = \emptyset$ for $i \neq i'$. Let
$$
C_{i,l} := W_{i,l} \cap h_{i}^{-1}(U_i \times V_{l-i+1}).
$$
It is clear that $C_{i,l}$ satisfies $h_i^{-1}(Z_{i,l}) \subseteq C_{i,l} \subseteq h_{i}^{-1}(U_i \times V_{l-i+1})$ and $C_{i,l} \cap C_{i',l} = \emptyset$ for $i \neq i'$. Let 
$$
    C_l:=\coprod_{i=1}^{l} C_{i,l}.
$$
for $1\leq l \leq m+n-1$. Note that $Q_l \subseteq C_l$. Since $h_i$ are induced by $f_i$, the composition of
\[
\begin{tikzcd}
    C_{i,l} \arrow[r, hook]  & h_{i}^{-1}(U_i \times V_{l-i+1}) \arrow[r, "h_i"] &  U_i \times V_{{l-i+1}} \arrow[r,"s_i \times t_{l-i+1}"] & R_i^{I} \times F^I ,
\end{tikzcd} 
\]    
and
\[
\begin{tikzcd}
   R_i^{I} \times F^I \arrow[r, "\cong"] & (R_i \times F)^I \arrow[r, "(f_i^{-1})^I"] \arrow[r, swap, "\cong"] & (p^{-1}(R_i))^I \arrow[r, hook] &  E^I \
\end{tikzcd} 
\]
gives a sequential motion planner on $C_{i,l}$ for each $1 \leq i \leq m$ and $1 \leq l \leq m+n-1$. As $C_l$ is a disjoint union of open sets $C_{i,l}$, we get a sequential motion planner on $C_l$ for each $1\leq l \leq m+n-1$. Furthermore,
     \begin{align*}
        \bigcup_{l=1}^{m+n-1} Q_l & = \bigcup_{l=1}^{m+n-1} \bigcup_{i=1}^{l} h_i^{-1}((X_i\setminus X_{i-1}) \times (Y_{l-i+1} \setminus Y_{l-i}))\\
        & = \bigcup_{i=1}^{m+n-1} \bigcup_{l=i}^{m+n-1}  h_i^{-1}((X_i\setminus X_{i-1}) \times (Y_{l-i+1} \setminus Y_{l-i}))\\
        & = \bigcup_{i=1}^{m+n-1} h_i^{-1}((X_i\setminus X_{i-1}) \times Y_{m+n-i})\\
        & = \left(\bigcup_{i=1}^{m} h_i^{-1}((X_i\setminus X_{i-1}) \times Y_{m+n-i})\right) \cup \left(\bigcup_{i=m+1}^{m+n-1} h_i^{-1}((X_i\setminus X_{i-1}) \times Y_{m+n-i})\right)\\
        & = \left(\bigcup_{i=1}^{m} h_i^{-1}((X_i\setminus X_{i-1}) \times F)\right) \cup \left(\bigcup_{i=m+1}^{m+n-1} h_i^{-1}(\emptyset \times Y_{m+n-i})\right)\\
        & = \bigcup_{i=1}^{m} h_i^{-1}((X_i\setminus X_{i-1}) \times F) \\
        & = \bigcup_{i=1}^{m} (p^k)^{-1}(X_i\setminus X_{i-1}) \\
        & =  (p^k)^{-1}(X_m) \\
        & =  (p^k)^{-1}(B^k) \\
        & = E^k.
    \end{align*} 

Hence, $E^k = \bigcup_{l=1}^{m+n-1} Q_l \subseteq \bigcup_{l=1}^{m+n-1} C_l$ implies $\{C_l\}_{i=1}^{m+n-1}$ is an open cover of $E^k$ with sequential motion planners. Hence, $\TC_k(E) \leq m+n-1$.
\end{proof}

\begin{remark}
\hfill
    \begin{itemize}
        \item[(i)] In \Cref{additive inequality for TC for fibre bundles}, the local trivializations $h_i$'s have to be the one defined in \Cref{def: induced trivialization for TC}, otherwise the sequential motion planners for $E$ defined via compositions in \Cref{additive inequality for TC for fibre bundles} can fail to be sections of $e_{k,E}$ which was an oversight in \cite[Theorem 2.2]{DaundSarkargpps}.

        \item[(ii)] In \Cref{additive inequality for TC for fibre bundles}, it is enough to assume that $(U_i \times \overline{U_{i'}}) \times V_j$ is invariant under $h_{i'} \circ h_i^{-1}$ for all $1 \leq i,i' \leq m$ and $1 \leq j \leq n$.
        
        \item[(iii)] Suppose that $F \hookrightarrow E \xrightarrow{p} B$ is the trivial fibre bundle. 
        Suppose that $\TC_k(B) = m$ and $\TC_k(F) = n$ with sequential motion planning covers $\{U_1,\dots,U_m\}$ and $\{V_1. \dots, V_n\}$. 
        If $f\colon E \to B \times F$ is a global trivialization for $p$, then we can take $f_i$'s to be $f$ for all $i$. 
        Then the condition (i) of \Cref{additive inequality for TC for fibre bundles} is satisfied since $R_i=B^k$ and the condition (ii) is satisfied since $h_{i'} \circ h_i^{-1}$ is the identity map. Hence, $\TC_k(E) \leq \TC_k(B) + \TC_k(F) -1$ and we recover the product inequality in \cite[Proposition 3.11]{Rudyak2014}.
    \end{itemize}
\end{remark}

\begin{remark}
\hfill
\label{remark for additive inequality of TC for fibre bundles}
\begin{itemize}
        \item[(i)] If the $R_i$'s in \Cref{additive inequality for TC for fibre bundles} are closed, then we have $\overline{U_i} \subseteq R_i^k$. Hence, $h_i$'s can be extended to local trivializations over $\overline{U_i}$ as we can take the extension to be $\left .f_i^k\right|_{(p^{k})^{-1}(\overline{U_i})}$.
        \item[(ii)] Suppose that $G$ is the structure group of the fibre bundle $F \hookrightarrow E \to B$, and the local trivializations $f_i \colon p^{-1}(R_i) \to R_i \times F$ in Theorem \ref{additive inequality for TC for fibre bundles} form a $G$-atlas of $F \hookrightarrow E \to B$, i.e, the transition maps $f_i \circ f_{i'}^{-1} \colon (R_i \cap R_{i'}) \times F \to (R_i \cap R_{i'}) \times F$ are given by 
        $$
            (f_i \circ f_{i'}^{-1})(x,y) = (x,t_{ii'}(x)y)
        $$
        where $t_{ii'} \colon R_{i}\cap R_{i'}\to G$ are continuous maps. 
        Then the transition maps $$f^k_i \circ (f^k_{i'})^{-1} \colon (R^k_i \cap R^k_{i'}) \times F^k \to (R^k_i \cap R^k_{i'}) \times F^k$$ are given by 
        $$
            (f^k_i \circ (f^k_{i'})^{-1})(x_1,\dots,x_k,y_1,\dots,y_k) = (x_1,\dots,x_k,t_{ii'}(x_1)y_1,\dots,t_{ii'}(x_k)y_k).
        $$
        In particular, if $V$ is a $G^k$-invariant subset of $F^k$, then $f^k_i \circ (f^k_{i'})^{-1}$ preserves $W \times V$ where $W$ is any subset of $R^k_i \cap R^k_{i'}$.
    \end{itemize}
\end{remark}

\begin{corollary}\label{cor:seqtc-ub-using-gcat}
Let $F \hookrightarrow E \xrightarrow{p} B$ be a fibre bundle with structure group $G$ where $E^k$ is a completely normal space. 
Let $\{U_1, \dots, U_m\}$ be an open cover of $B^k$ with sequential motion planners $s_i \colon U_i \to B^I$.
If there exists a closed cover $\{R_1, \dots, R_m\}$ of $B$ with local trivializations $f_i \colon p^{-1}(R_i) \to R_i \times F$ of $F \hookrightarrow E \xrightarrow{p} B$ such that $f_i$'s form a $G$-atlas of $F \hookrightarrow E \xrightarrow{p} B$ and $s_i(U_i)\subseteq  R_i^I$, then 
$$
\TC_k(E) \leq m+\TC_{k,G}^{\#,*}(F)-1.
$$
In particular, $\TC_k(E) \leq m+\ct_{G^k}^{\#}(F^k)-1$ and $\TC_k(E) \leq m+\TC_{k,G}^*(F)-1$.
\end{corollary}

\begin{proof}
Suppose that $\TC_{k,G^k}^{\#,*}(F)=n$ and $\{V_j\}_{j=1}^{n}$ be a $G^k$-invariant open cover of $F^k$ with sequential motion planners. By Remark \ref{remark for additive inequality of TC for fibre bundles}, we have $h_i \circ h_{i'}^{-1}$ preserves $(\overline{U_i} \cap \overline{U_{i'}}) \times V_j$ as $\overline{U_i} \subseteq R_i^k$ and $h_i \circ h_{i'}^{-1}$ is restriction of $f^k_i \circ (f^k_{i'})^{-1}$. Thus, by \Cref{additive inequality for TC for fibre bundles}, it follows that $\TC_k(E) \leq m+n-1$.
\end{proof}

\section{LS category of fibre bundles}
\label{sec:ls-category-fibre-bundles}

Let $F \hookrightarrow E \xrightarrow{p} B$ be a fibre bundle. 
In \cite[Theorem 2.6]{Naskar}, Naskar and the third author provided an additive upper bound on the LS category of $E$. 
Due to an oversight in the Lemma preceding it, we state the modified versions of \cite[Lemma 2.5]{Naskar} and \cite[Theorem 2.6]{Naskar} in \Cref{lemma: induced trivialization for cat} and \Cref{thm: additive upper bound for ls category}, respectively.
We omit the proof of \Cref{thm: additive upper bound for ls category}, since it follows the same line of argument as \cite[Theorem 2.6]{Naskar} and \Cref{additive inequality for TC for fibre bundles}. 
Finally, we provide an additive upper bound on $\ct(E)$ in terms of $\ct_{G}^{\#}(F)$ if the fibre bundle $p$ has the structure group $G$.

\begin{lemma}
\label{lemma: induced trivialization for cat}
    Let $F \hookrightarrow E \xrightarrow{p} B$ be a fibre bundle over a paracompact space $B$. 
    If $U \subseteq R \subseteq B$ is such that $U$ is a categorical subset of $R$ and $R$ is a categorical subset of $B$, then for any trivialization $\phi \colon p^{-1}(R) \to R \times F$ of $p$ over $R$ and any categorical subset $V$ of $F$, we have the preimage $\phi^{-1}(U \times V)$ is categorical in $E$.
\end{lemma}

\begin{proof} 
    Since the inclusion $R \hookrightarrow B$ is null-homotopic and $B$ is paracompact, it follows that $E$ is trivial over $R$, see \cite[Proposition 4.62]{hatcher}. 
    Suppose $\phi \colon p^{-1}(R) \to R \times F$ is a trivialization of $p$ over $R$. 
    As $U \times V$ is categorical in $R \times F$, it follows that $ \phi^{-1}(U \times V)$ is categorical in $p^{-1}(R)$.
    Hence, $\phi^{-1}(U \times V)$ is categorical in $E$.
\end{proof}

\begin{remark}
\label{rem: induced trivialization over contractible subset}
    If $R$ is a contractible subset of a space $B$, then any chain of inclusions $U \subseteq R \subseteq B$ satisfies the condition that $U$ is a categorical subset of $R$ and $R$ is a categorical subset of $B$, as required in \Cref{lemma: induced trivialization for cat}.
\end{remark}

\begin{theorem}
\label{thm: additive upper bound for ls category}
    Let $F \hookrightarrow E \xrightarrow{p} B$ be a fibre bundle, where $E$ is a completely normal, path-connected space and $B$ is a paracompact space. 
    Let $\{V_1,\dots,V_n\}$ be a categorical open cover of $F$, and let $\{U_1, \dots, U_m\}$ be an open cover of $B$ such that, for each $i \in \{1,\dots,m\}$, there exists a closed subset $R_i$ of $B$ that contains $U_i$ such that $U_i$ is categorical in $R_i$ and $R_i$ is categorical in $B$. 
    If, for each $i \in \{1,\dots,m\}$, there exist local trivializations $\phi_i\colon p^{-1}(R_i) \to R_i \times F$ of $p$ over $R_i$ such that $(\overline{U_i} \cap \overline{U_{i'}}) \times V_j$ is invariant under $\phi_{i'} \circ \phi_i^{-1}$ for all $1 \leq i , i' \leq m$ and $1 \leq j \leq n$, then $\ct(E) \leq m+n-1$.
\end{theorem}

\begin{corollary}
\label{cor: additive ub on ct of fb}
Let $F \hookrightarrow E \xrightarrow{p} B$ be a fibre bundle with structure group $G$, where $E$ is a completely normal, path-connected space and $B$ is a paracompact space. 
Let $\{U_1, \dots, U_m\}$ be an open cover of $B$ such that, for each $i \in \{1,\dots,m\}$, there exists a closed subset $R_i$ of $B$ that contains $U_i$ such that $U_i$ is categorical in $R_i$ and $R_i$ is categorical in $B$. 
If there exist trivializations $\phi_i\colon p^{-1}(R_i) \to R_i \times F$ of $p$ over $R_i$ such that the restriction of the $\phi_i$'s over $\overline{U_i}$'s form a $G$-atlas of $p$, then $\ct(E) \leq m+\ct^{\#}_{G}(F)-1$.
\end{corollary}

\begin{proof}
    Suppose that $\ct^{\#}_{G}(F)=n$ and $\{V_j\}_{j=1}^{n}$ be a $G$-invariant categorical open cover of $F$. 
    Note that the transition maps $\phi_{i'} \circ \phi_i^{-1}\colon (\overline{U_i} \cap \overline{U_{i'}}) \times F \to (\overline{U_i} \cap \overline{U_{i'}}) \times F$ are given by
    $$
    \phi_{i'} \circ \phi_i^{-1}(x,y) = (x,\tau_{i'i}(x)y)
    $$
    for some continuous maps $\tau_{ii'}\colon  \overline{U_i} \cap \overline{U_{i'}} \to G$ for all $1 \leq i,i' \leq m$. Hence, $(\overline{U_i} \cap \overline{U_{i'}}) \times V_j$ is invariant under $\phi_{i'} \circ \phi_i^{-1}$ for all $1 \leq i , i' \leq m$ and $1 \leq j \leq n$, since $V_j$'s are $G$-invariant. Thus, by \Cref{thm: additive upper bound for ls category}, it follows $\ct(E) \leq m+n-1$.
\end{proof}

\section{Application to generalized projective product spaces}\label{sec:applica-gpps}

Let $M$ and $N$ be topological spaces with involutions $\tau \colon M \to M$ and $\sigma \colon N \to N$ such that $\sigma$ is fixed-point free. Sarkar and Zvengrwoski in \cite{sarkargpps} introduced  the following identification spaces
\begin{equation}\label{eq: gpps}
X(M, N) := \frac{M \times N}{(x,y)\sim (\tau(x), \sigma(y))}.
\end{equation}
These identification spaces are called \emph{generalized projective product spaces}. Note that this class of spaces contains projective product spaces, introduced by Davis in \cite{Davis} and Dold manifolds by Dold in \cite{Dold}.

Let $\pi_2 \colon M \times N \to N$ be the projection map onto $N$, and $q_2 \colon N \to N/\left<\sigma\right>$ and ${q \colon M \times N \to X(M,N)}$ be the natural quotient maps. Then there exists a continuous map $\mathfrak{p} \colon X(M,N) \to N/\left<\sigma \right>$, given by $\mathfrak{p}\left(\left[x,y\right]\right) = q_2(y)$, such that the diagram
\[
    \begin{tikzcd}
        M\times N \arrow[r, "\pi_2"] \arrow[d, "q"] & N \arrow[d, "q_2"]       \\
        {X(M,N)} \arrow[r, "\mathfrak{p}", dotted]  & N/\langle \sigma \rangle
    \end{tikzcd}
\]
commutes, since $ q_2(\pi_2(\tau(x),\sigma(y))) = q_2(\sigma(y)) = q_2(y) = q_2(\pi_2(x,y))$.

\begin{proposition}
\label{tau-atlas for projective product spaces}
If $N$ is Hausdorff, then $q_2$ is a covering map. Moreover, 
\begin{equation}\label{eq: gpps fb}
    M \xhookrightarrow{} X(M, N) \stackrel{\mathfrak{p}}\longrightarrow N/\left<\sigma \right>,
\end{equation}
is a fibre bundle with structure group $\left< \tau \right>$, and a $\left< \tau \right>$-atlas for $\mathfrak{p}$ is given over all evenly covered open subsets of $N/\left< \sigma \right>$ with respect to $q_2$.
\end{proposition}

\begin{proof}
Suppose that $z_0 \in N/\left<\sigma \right>$. Since the action of $\left<\sigma \right>$ on $N$ is free and $N$ is Hausdorff, by \cite[Exercise 23, Page 81 and Proposition 1.40]{hatcher}, $q_2$ is a $2$-sheeted covering map. Hence, there exists an open set $U$ of $N/\left< \sigma \right>$ such that $q_2^{-1}(U) = V \amalg \sigma(V)$ and $\left.q_2\right|_{V} \colon V \to U$ is a homemorphism, where $V$ is an open subset of $N$. As 
$$
    q^{-1}(\mathfrak{p}^{-1}(U)) 
        = \pi_2^{-1}(q_2^{-1}(U))
        = \pi_2^{-1}(V \amalg \sigma(V)) 
        = \left( M \times V \right) \amalg \left( M \times \sigma(V) \right)
$$
and $q$ is surjective, it follows 
$$
    \mathfrak{p}^{-1}(U)
        = q(q^{-1}(p^{-1}(U)))
        = \frac{\left( M \times V \right) \amalg \left( M \times \sigma(V) \right)}{(x,y)\sim (\tau(x), \sigma(y))}
$$
Define $\phi \colon \left( M \times V \right) \amalg \left( M \times \sigma(V) \right) \to M \times U$ by 
$$
\phi(x,y) = 
    \begin{cases}
        (x,q_2(y)) & \text{ if }(x,y) \in M\times V,\\
        (\tau(x),q_2(y)) & \text{ if }(x,y) \in M\times \sigma(V).
    \end{cases}
$$
The map $\phi$ is well-defined since $V \cap \sigma(V) = \emptyset$. Moreover, $\phi$ satisfies $\phi(x,y)=\phi(\tau(x),\sigma(y))$, and hence induces 
$$
\Phi_U \colon \frac{\left( M \times V \right) \amalg \left( M \times \sigma(V) \right)}{(x,y)\sim (\tau(x), \sigma(y))} \to M \times U.
$$
Define 
$$
    \Psi_U \colon M \times U \to \frac{\left( M \times V \right) \amalg \left( M \times \sigma(V) \right)}{(x,y)\sim (\tau(x), \sigma(y))}
$$
as the composition of maps
\[
\begin{tikzcd}
    M \times U \arrow[r,"\text{id}_M \times (\left.q_2\right|_{V})^{-1}"]  &[3em] M \times V \arrow[r, hook] &  \left( M \times V \right) \amalg \left( M \times \sigma(V) \right) \arrow[r,"\mathfrak{p}"] & \frac{\left( M \times V \right) \amalg \left( M \times \sigma(V) \right)}{(x,y)\sim (\tau(x), \sigma(y))}.
\end{tikzcd} 
\] 
Then $\Phi_U$ and $\Psi_U$ are inverses of each other and define a trivialization of $\mathfrak{p}$. 

Suppose that $U'$ is another evenly covered neighborhood with $q_2^{-1}(U')=V' \amalg \sigma(V')$ such that $U \cap U' \neq \emptyset$. If $q_2^{-1}(U \cap U') = (V \cap V') \amalg (\sigma(V) \cap \sigma(V'))$, then the transition function
$
    \Phi_{U} \circ \Psi_{U'} \colon M \times (U \cap U') \to M \times (U \cap U')
$
is given by
$$
    \Phi_{U} \left(\Psi_{U'}(x,z)\right)
        = \Phi_{U} \left(\left[x,(\left.q_2\right|_{V'})^{-1}(z)\right] \right)
        = \left(x,q_2((\left.q_2\right|_{V'})^{-1}(z))\right) 
        = (x,z).
$$
If $q_2^{-1}(U \cap U') = (V \cap \sigma(V')) \amalg (\sigma(V) \cap V')$, then
${\Phi_{U} \circ \Psi_{U'} \colon M \times (U \cap U') \to M \times (U \cap U')}$
is given by
$$
    \Phi_{U} \left(\Psi_{U'}(x,z)\right)
        = \Phi_{U} \left(\left[\tau(x),(\left.q_2\right|_{V'})^{-1}(z)\right] \right)
        = \left(\tau(x),q_2((\left.q_2\right|_{V'})^{-1}(z))\right) 
        = (\tau(x),z).
$$
Thus, the family $\{\Phi_U\}$ defines a $\left< \tau \right>$-atlas for $\mathfrak{p}$.
\end{proof}

\begin{remark}
    Any set (need not be open) which is evenly covered by the covering map $q_2$ can also be taken as a part of $\left<\tau\right>$-atlas for the fibre bundle $\mathfrak{p}$.
\end{remark}

\begin{proposition}
\label{prop: cat of gpps with base S^n}
    Let $M$ be a path-connected topological space.
    Suppose that $\tau \colon M \to M$ is an involution on $M$ and $\sigma \colon S^n \to S^n$ is the antipodal involution on $S^n$ given by $\sigma(z)=-z$.
    If $X(M,S^n)$ is completely normal, then
    $$
        \ct(X(M,S^n)) \leq \ct_{\left< \tau\right>}^{\#}(M)+n.
    $$
\end{proposition}

\begin{proof}
We will use the same notation as \Cref{tau-atlas for projective product spaces}. 
Note that $q_2 \colon S^n \to S^n/\left<\sigma\right> = \mathbb{R}P^n$ is the natural quotient map. 
Let $U_r$ be the open hemisphere with center (or top) at the point $e_r=(0,\dots,0,1,0,\dots,0) \in S^n$, where the $1$ is in the $r^{\text{th}}$ position. 
Let $C_r$ be a closed spherical cap contained inside $U_r$ that approximates $U_r$, i.e., the height of the spherical cap is almost the radius of the sphere (for example, we can take the polar angle to be $89^{\mathrm{o}}$). 
Then $\{q_2(C_r)\}_{r=1}^{n+1}$ form a closed cover of $\R P^n$ and $\{q_2(C_r^{\mathrm{o}})\}_{r=1}^{n+1}$ form an open cover of $\R P^n$, where $C_r^{\mathrm{o}}$ is the interior of $C_r$.
As $q_2(C_r)$'s are evenly covered by the covering map $q_2$, by \Cref{tau-atlas for projective product spaces}, it follows that $\Phi_{q_2(C_r)} \colon \mathfrak{p}^{-1}(q_2(C_r)) \to M \times q_2(C_r)$, for $r=1,\dots,n+1$, form a $\left<\tau \right>$-atlas for $\mathfrak{p}$. 
Since $q_2(C_r)$ is contractible, the chain of inclusions $q_2(C_r^{\mathfrak{o}}) \subseteq q_2(C_r) \subseteq \R P^n$ satisfies the condition that $q_2(C_r^{\mathfrak{o}})$ is categorical in $q_2(C_r)$ and $q_2(C_r)$ is categorical in $\R P^n$, see \Cref{rem: induced trivialization over contractible subset}.
Thus, the hypotheses of \Cref{cor: additive ub on ct of fb} are satisfied since $\overline{C_r^{\mathrm{o}}} = C_r$, and we obtain
$$
    \ct(X(M,S^n)) \leq (n+1)+\ct_{\left<\tau\right>}^{\#}(M) - 1 = \ct_{\left<\tau\right>}^{\#}(M)+n,
$$
which gives the desired upper bound.
\end{proof}

\begin{proposition}
\label{prop: cat and tck of gpps with base S1}
    Let $M$ be a path-connected topological space.
    Suppose that $\tau \colon M \to M$ is an involution on $M$ and $\sigma \colon S^1 \to S^1$ is the antipodal involution on $S^1$ given by $\sigma(z)=-z$.
    If $X(M,S^1)$ is Hausdorff 
    and $X(M,S^1)^k$ is completely normal, then
    $$
        \TC_k(X(M,S^1)) \leq k + \TC_{k,\left< \tau\right>}^{\#,*}(M) \leq k\,\ct_{\left< \tau\right>}^{\#}(M)+1.
    $$
\end{proposition}

\begin{proof}
We will use the same notation as \Cref{tau-atlas for projective product spaces}. Note that $q_2 \colon S^1 \to S^1/\left<\sigma\right> = \mathbb{R}P^1$ is the natural quotient map. If $\zeta \colon \mathbb{R}P^1 \to S^1$ is the homeomorphism given by $\zeta(q_2(z))=z^2$, then we can replace $\mathfrak{p} \colon X(M,S^1) \to \mathbb{R}P^1$ with $\zeta \circ \mathfrak{p} \colon X(M,S^1) \to S^1$ to get a fibre bundle
\[
\begin{tikzcd}[column sep = 2.5 em]
    M  \arrow[hook, r] & X(M,S^1) \arrow[r, "\zeta \circ \mathfrak{p}"] & S^1.
\end{tikzcd}
\]
For $i=1,\dots,k+1$, let $R_i$ be closed subsets of $S^1$ whose complement is given by 
$$
R_i^c:=\left\{e^{\sqrt{-1}\,\theta} \in S^1 \;\Big|\; \frac{2(i-1)\pi}{k+1} < \theta < \frac{(2i-1)\pi}{k+1}  \right\}.
$$
Observe that $\{R_1, \dots,R_{k+1}\}$ form a closed cover of $S^1$. 
If $U_i$ is the interior of $R_i$, then each point $x \in S^1$ lies in at least $k$ elements from the open cover $\{U_1, \dots, U_{k+1}\}$.
Hence, $\{U_1^k, \dots, U_{k+1}^k\}$ forms an open cover of $(S^1)^k$.
As $U_i$ is contractible, there exists a global section $s_i \colon U_i^k \to U_i^I$ of $e_{k,U_i} \colon U_i^I \to U_i^k$. 
In particular, $s_i$ is a section of $e_{k,S^1}$ over $U_i^k$, via the inclusion $U_i^I \hookrightarrow (S^1)^I$. Thus, $\{U_1^k, \dots,U_{k+1}^k\}$ form a sequential motion planning open cover of $(S^1)^k$. 
Moreover, the section $s_i$ satisfies $s_i(U_i) \subseteq R_i^I$, via the inclusion $U_i^I \hookrightarrow R_i^I$. 
As $R_i$'s are evenly covered under the covering map $\zeta \circ q_2 \colon  S^1 \to S^1$, by \Cref{tau-atlas for projective product spaces}, it follows that $\Phi_{R_i} \colon (\zeta \circ \mathfrak{p})^{-1}(R_i) \to M \times R_i$, for $i=1,\dots, k$, form a $\left<\tau \right>$-atlas for $\zeta \circ \mathfrak{p}$. 
Thus, the conditions of \Cref{cor:seqtc-ub-using-gcat} are satisfied and we have
$$  \TC_k(X(M,S^1))
        \leq (k+1) + \TC_{k,\left< \tau\right>}^{\#,*}(M) - 1
        = k + \TC_{k,\left< \tau\right>}^{\#,*}(M).
$$    
Thus, it follows that
$$
    \TC_k(X(M,S^1)) 
    \leq k + \ct_{\left< \tau\right>^k}^{\#}(M^k)
    \leq k + k(\ct_{\left< \tau\right>}^{\#}(M)-1)+1 
    = k\,\ct_{\left< \tau\right>}^{\#}(M)+1
$$
by \Cref{prop:ct-leq-strongtc-wrtG} and \Cref{cor:new-Gcat-prod-ineq-for-X^k}.
\end{proof}

\begin{remark}
    We can generalize the previous proposition to $\TC_{k}(X(M,S^n))$, if we can find a closed cover $\{R_1,\dots,R_{s+1}\}$, where $s \geq k$, of $\R P^n$ such that 
\begin{enumerate}
    \item each $R_i$ is evenly covered under the quotient map $S^n \to \R P^n$,
    \item $\mathcal{U}=\{U_1, \dots, U_{s+1}\}$ forms an open cover of $\R P^n$, where $U_i$ is the interior of $R_i$,
    \item each $U_i$ is contractible, and
    \item each point in $\R P^n$ lies in at least $s$ elements from $\mathcal{U}$. 
\end{enumerate}
Then $\{U_1^k,\dots,U_{s+1}^k\}$ forms an open cover of $(\R P^n)^k$ satisfying the hypotheses of \Cref{cor:seqtc-ub-using-gcat}. 
Hence, we obtain $\TC_{k}(X(M,S^n)) \leq s + \TC^{\#,*}_{k,\langle \tau\rangle}(M)$.
\end{remark}

For each $r$-tuple $\overline{n} = (n_1,\dots,n_r)$ of positive integers, with $n_1 \geq n_2 \geq \dots \geq n_r$, the projective product space $P_{\overline{n}}$ is defined as
$$
    P_{\overline{n}} 
        := \left( S^{n_1} \times \dots \times S^{n_r}\right)/(x_1,\dots,x_r) \sim (-x_1,\dots,-x_r).
$$
These spaces were introduced and studied by Davis in \cite{Davis}.

\begin{lemma}
\label{lemma: cat-wrt-G of product of spheres}
Let $G_j$ be a finite group acting on $S^{n_j}$, where $n_j \geq 1$ for $j \in \{1,\dots,r\}$. If $G := G_1 \times \dots \times G_r$ is the product group acting on $S^{n_1}\times \dots \times S^{n_r}$ componentwise, then
$$
    \ct_{G}^{\#}\left(\prod_{j=1}^r S^{n_j}\right) = r+1.
$$
\end{lemma}

\begin{proof}
By \Cref{example:new-Gcat(S^n)}, we have $\ct_{G_j}^{\#} \left(S^{n_j}\right) = 2$. Therefore, from \Cref{cor:new-Gcat-prod-ineq-for-X^k}, we get the following inequality 
\[
    \ct_{G}^{\#}\left(\prod_{j=1}^rS^{n_j}\right)
        \leq \left(\sum_{j=1}^r\ct_{G_j}^{\#}(S^{n_j})\right)-(r-1)
        = r+1.
\]
As $\ct \left(\prod_{j=1}^rS^{n_j}\right) = r+1$, it follows $\ct_{G}^{\#} \left(\prod_{j=1}^rS^{n_j}\right) = r+1$.
\end{proof}

The following corollary recovers \cite[Theorem 1.2]{Vandembroucq}, as the lower bound on $\ct(P_{\overline{n}})$ can be obtained from the description of the cohomology ring of $P_{\overline{n}}$ in \cite[Theorem 2.1]{Davis}; see \cite[Proposition 2.3]{Vandembroucq} for the lower bound. 

\begin{corollary}
    Let $\overline{n} = (n_1,n_2,\dots,n_r)$, where $n_1 \geq \dots \geq n_r \geq 1$. Then
    $
        \ct(P_{\overline{n}}) \leq n_r + r.
    $
\end{corollary}

\begin{proof}
    Let $\tau_j$ be the antipodal involution on $S^{n_j}$ for $j=1,\dots,r$. 
    Let $M$ be the product space $S^{n_1} \times \dots \times S^{n_{r-1}}$, and let $\tau$ be the product involution $\tau_1 \times \dots \times \tau_{r-1}$ on $M$. 
    Then $P_{\overline{n}} = X(M,S^{n_r})$.
    Hence,
    \begin{align*}
        \ct(P_{\overline{n}}) 
            \leq \ct_{\langle \tau \rangle}^{\#}(S^{n_1} \times \dots \times S^{n_{r-1}}) + n_r
            \leq r + n_r.
    \end{align*}
    by \Cref{prop: cat of gpps with base S^n} and \Cref{lemma: cat-wrt-G of product of spheres}.
\end{proof}

\begin{corollary}
    Let $\overline{n} = (n_1,n_2,\dots,n_{r-1},1)$, where $n_1 \geq \dots \geq n_{r-1} \geq n_r =1$. Then
    \begin{equation}
        \label{eq: TC of proj-prod-spaces}
        \TC_k(P_{\overline{n}}) 
            \leq k + \sum_{j=1}^{r-1} \TC^{\#,*}_{k,\langle \tau_j \rangle}(S^{n_j}) - (r-2) 
            \leq kr+1,
    \end{equation}
    where $\tau_j$ is the antipodal involution on $S^{n_j}$ for $j=1,\dots,r$.
\end{corollary}

\begin{proof}
    Let $\tau$ be the product involution $\tau_1 \times \dots \times \tau_{r-1}$ on $S^{n_1} \times \dots \times S^{n_{r-1}}$. 
    Then
    \begin{align*}
        \TC_k(P_{\overline{n}}) 
            \leq k + \TC^{\#,*}_{k,\langle \tau \rangle}(S^{n_1} \times \dots \times S^{n_{r-1}})
            \leq k +\sum_{j=1}^{r-1} \TC^{\#,*}_{k,\langle \tau_j \rangle}(S^{n_j}) - (r-2).
    \end{align*}
    by \Cref{prop: cat and tck of gpps with base S1} and \Cref{prop:prod-ineq-tck-wrt_G}. 
    Then 
    \begin{equation}
        \TC_k(P_{\overline{n}}) \leq k + (r-1)(k+1) - (r-2) \leq kr+1.
    \end{equation}
    This follows from the fact that $\TC^{\#,*}_{k,\langle \tau_j \rangle}(S^{n_j}) \leq k+1$, see \Cref{example:TC(S^n)-wrt-G}. 
\end{proof}

\begin{remark} The inequality \eqref{eq: TC of proj-prod-spaces} may be compared with earlier results in the literature as follows:
    \begin{enumerate}
        \item Suppose $n_r=1$ and $n_j$ is even for some $j = 1,\dots,r-1$. For $k=2$, the inequality \eqref{eq: TC of proj-prod-spaces} offers a stronger upper bound on $\TC_2(P_{\overline{n}})$ than the one given in \cite[Theorem 3.8]{eqtcprodineq}, which gives $\TC_2(P_{\overline{n}}) \leq 2r+2s$, where $s$ is the number of spheres $S^{n_j}$ with $n_j$ even.
        \item Suppose $n_r=1$.
        For $k=2$, the inequality \eqref{eq: TC of proj-prod-spaces} is weaker than the upper bound given in \cite[Theorem 1.3]{Vandembroucq}, which gives $\TC_2(P_{\overline{n}}) \leq r+s+1$, where $s$ is the number of spheres $S^{n_j}$ with $n_j$ even.
    \end{enumerate}
\end{remark}

Consider an involution $\tau_j$ on $S^{n_j} \subset \R^{n_j+1}$ defined as follows:
\begin{equation}\label{eq: invo prodsphere}
    \tau_j( (y_1, \ldots, y_{p_j}, y_{p_j+1}, \ldots, y_{n_j+1}) )
        := (y_1, \ldots, y_{p_j}, -y_{p_j+1}, \ldots, -y_{n_j+1}),  
\end{equation}
for some $0 \leq p_j \leq n_j$ and $1\leq j\leq r$. Note that if $p_j=0$, then $\tau_j$ acts antipodally on $S^{n_j}$; and if $p_j=n_j$, then $\tau_j$ is a reflection across the hyperplane $y_{n_j+1}=0$ in $\R^{n_j+1}$. Then we have $\Z_2$-action on the product $S^{n_1}\times \dots \times S^{n_{r}}$ via the product involution 
\begin{equation}\label{eq: invo prod taui}
    \tau :=\tau_1\times \dots \times \tau_{r}\,.
\end{equation}

Let $N$ be a topological space with a free involution $\sigma$. Consider the identification space: 
\begin{equation}\label{eq:prodsphere_N}
    X((n_1, p_1), \ldots, (n_r,p_r), N) 
        :=\frac{ S^{n_1}\times \cdots \times S^{n_r}\times N}{({\bf x}_1, \dots, {\bf x}_{r}, y)\sim (\tau_1({\bf x}_1),\dots ,\tau_r({\bf x}_r), \sigma(y))},
\end{equation}
where $\tau_j$ is a reflection defined as in \eqref{eq: invo prodsphere} for $1\leq j\leq r$. So, $X((n_1, p_1), \ldots, (n_r,p_r), N)$ is a generalized projective product space.

\begin{theorem}
\label{thm:tck-main-example}
Let $1 \leq p_j \leq n_j$ for $j=1, \ldots, r$. Then
\begin{equation}
\label{eq:cat-gpps1}
    \mathrm{cup}_{\mathbb{Z}_2}(N/\left< \sigma \right>)+r + 1 \leq  \ct(X((n_1, p_1), \ldots, (n_r,p_r), N)),\, \text{and}
\end{equation}  
\begin{equation}
\label{eq:seqtc-gpps1}
    \zl_k(N/\left< \sigma \right>) +(k-1)r+1 \leq  \TC_k(X((n_1, p_1), \ldots, (n_r,p_r), N)).
\end{equation}  
In particular, 
\begin{itemize}
    \item if $N=S^n$ and $\sigma \colon S^n \to S^n$ is the antipodal involution, then
\begin{equation}
\label{eq:cat-gpps1S1}
      \ct(X((n_1, p_1), \ldots, (n_r,p_r), S^n)) = r+n+1,\,\text{and}   
\end{equation}
\item if $N=S^1$ and $\sigma \colon S^1 \to S^1$ is the antipodal involution, then
\begin{equation}
\label{eq:seqtc-gpps1S1}
    (k-1)(r+1)+1\leq  \TC_k(X((n_1, p_1), \ldots, (n_r,p_r), S^1))\leq k(r+1)+1.   
\end{equation}
\end{itemize}
\end{theorem}

\begin{proof}
The inequality \eqref{eq:cat-gpps1} follows from \cite[Proposition 4.3]{DaundSarkargpps}.
Hence, by \eqref{eq:cat-gpps1}, we get 
$$
    n+r+1 \leq \ct(X((n_1, p_1), \ldots, (n_r,p_r), S^n))
$$
since $\mathrm{cup}_{\mathbb{Z}_2}(\R P^n) = n$.
Thus, the equality \eqref{eq:cat-gpps1S1} follows from \Cref{prop: cat of gpps with base S^n} and \Cref{lemma: cat-wrt-G of product of spheres} since 
$$
    \ct(X((n_1, p_1), \ldots, (n_r,p_r), S^n)) 
        \leq \ct_{\left< \tau \right>}^{\#}(\prod_{j=1}^r S^{n_j}) + n 
        = (r+1)+n = r+n+1.
$$

We now show the inequalities concerning the sequential topological complexity starting with the inequality \eqref{eq:seqtc-gpps1}. Note \cite[Theorem 4.1]{DaundSarkargpps} states that
\[
    H^{*}(X((n_1, p_1), \ldots, (n_r,p_r), N);\mathbb{Z}_2) \cong H^{*}(N/\left< \sigma \right>;\mathbb{Z}_2) \otimes_{\mathbb{Z}_2} \Lambda(\beta_1,\dots,\beta_r) 
\]
where $\Lambda(\beta_1,\dots,\beta_r)$ is the exterior algebra on $r$ generators over $\mathbb{Z}_2$.
Thus, we have
\[
    \zl_k(X((n_1, p_1), \ldots, (n_r,p_r), N))
        \geq \zl_k(N/\left< \sigma \right>)+\nil(\ker(\phi_k)),
\]
where
\[
  \phi_k:  \Lambda(\beta_1,\dots,\beta_r) \otimes_{\mathbb{Z}_2} \dots \otimes_{\mathbb{Z}_2} \Lambda(\beta_1,\dots,\beta_r)
        \to \Lambda(\beta_1,\dots,\beta_r)
\]
is the $k$-fold product homomorphism of $\Lambda(\beta_1,\dots,\beta_r)$. 

Let $y^i_j = 1 \otimes \dots \otimes 1 \otimes \beta_j \otimes 1 \otimes \dots \otimes 1 \in \Lambda(\beta_1,\dots,\beta_r)^{\otimes k}$, where $\beta_j$ occurs at $i^{\text{th}}$ position. 
Then $y^i_j + y^{i'}_j \in \ker(\phi_k)$ for all $1 \leq j \leq r$ and $1 \leq i,i' \leq k$, since $\phi_k(y^{i}_j+y^{i'}_j)=\beta_j + \beta_j = 0$. 
Consider the product
$$
    P := \prod_{j=1}^{r} \prod_{i=2}^{k} (y^{i-1}_j + y^i_j) \in \ker(\phi_k).
$$
This product is nonzero as this contains a nonzero term $\prod_{j=1}^{r} \prod_{i=2}^{k} y^i_j$ which can't be killed by any other term in the product $P$. 
Hence, $\nil(\ker(\phi_k)) \geq (k-1)r$.

Therefore, from \Cref{prop: zcl < TC}, we get the inequality \eqref{eq:seqtc-gpps1}. 
Now the left inequality of \eqref{eq:seqtc-gpps1S1} follows from the observation $\zl_k(\R P^1)=k-1$. 
The right inequality of \eqref{eq:seqtc-gpps1S1} follows from \Cref{prop: cat and tck of gpps with base S1} and \Cref{lemma: cat-wrt-G of product of spheres}.
\end{proof}

\begin{remark}
  Note that if $n_j=1=p_j$ for all $1\leq j\leq r$ and $N=S^1$ with the antipodal involution, then $X((n_1, p_1), \ldots, (n_r,p_r), N)$ is the $(r+1)$-dimensional Klein bottle. Then \eqref{eq:cat-gpps1S1} recovers Davis's result \cite[Corollary 2.3]{n-Klein-bottle}.  
\end{remark}

\section{Application to mapping tori}\label{sec:applica-mapping tori}

Suppose $f \colon M \to M$ is a homeomorphism of a topological space $M$. Let $M_f$ be the mapping torus corresponding to $f$, i.e., $M_f$ is the quotient space defined as
$$
M_f := \frac{M \times I}{(x,0) \sim (f(x),1)}.
$$

Suppose that $\pi_2 \colon M \times I \to I$ be the projection map on the second factor, and $q_2 \colon I \to S^1$ and ${q \colon M \times I \to M_f}$ be the natural quotient maps. 
Then there exists a continuous map ${\mathfrak{p} \colon M_f \to S^1}$, given by $\mathfrak{p}\left(\left[x,t\right]\right) = q_2(t)$, such that the diagram
\[
    \begin{tikzcd}
        M\times I \arrow[r, "\pi_2"] \arrow[d, "q"] & I \arrow[d, "q_2"]       \\
        M_f \arrow[r, "\mathfrak{p}", dotted]  & S^1
    \end{tikzcd}
\]
commutes, since $ q_2(\pi_2(f(x),1)) = q_2(1) = q_2(0) = q_2(\pi_2(x,0))$.

It is well known that $\mathfrak{p} \colon M_f \to S^1$ is a fibre bundle with fibre $M$. If $x_0 = q_2(0) = q_2(1)$, then a trivialization of $\mathfrak{p}$ over the open set $U_0 = S^1 \setminus \{x_0\}$ is given by
$$
\phi \colon M \times S^1 \setminus \{x_0\} \to q(M \times (0,1)) = \mathfrak{p}^{-1}(U_0),  \quad \phi(m,q_2(t)) = q(m,t).
$$
If $p$ is a point in $(0,1)$ and $U_p = S^1 \setminus \{q_2(p)\}$, then the map
$$
\psi_p \colon M \times S^1 \setminus \{q_2(p)\} \to q\left(M \times \left([0,p) \cup (p,1]\right)\right) = \mathfrak{p}^{-1}(U_p)
$$ 
defined by
$$
\psi_p(m,q_2(t)) =
\begin{cases}
    q(f^{-1}(m),t) & 0 \leq t <p,\\
    q(m,t) & p <  t \leq 1
\end{cases}
$$
is a trivialization of $\mathfrak{p}$ over the open set $U_p$.
Note that the transition maps $\phi^{-1} \circ \psi_p$ and $\psi_p^{-1} \circ \psi_{q}$, where we assume $p<q$ without loss of generality, are given by 
    $$
        (\phi^{-1} \circ \psi_p)(m,q_2(t)) = 
            \begin{cases}
                (f^{-1}(m),q_2(t)) & t \in (0,p),\\
                (m,q_2(t)) & t \in (p,1);
            \end{cases}
    $$
and
     $$
        (\psi_q^{-1} \circ \psi_p)(m,q_2(t)) = 
            \begin{cases}
                (m,q_2(t)) & t \in [0,p)\cup (q,1],\\
                (f(m),q_2(t)) & t \in (p,q);
            \end{cases}
    $$
respectively.

\begin{theorem}
\label{thm: cat of mapping torus}
    Suppose $f \colon M \to M$ is a homeomorphism of a topological space $M$ such that the corresponding mapping torus $M_f$ is completely normal. 
    If $\{V_1, \dots, V_n\}$ is an open cover of $M$ such that each $V_j$ is invariant under $f$, then $\ct(M_f) \leq n+1$.  
\end{theorem}

\begin{proof}
    Let $\{C_1,C_2\}$ be the closed cover of $S^1$ defined in \Cref{prop: cat and tck of gpps with base S1}. 
    Then the fibre bundle $M \hookrightarrow M_f \xrightarrow{\mathfrak{p}} S^1$ is trivial over $C_1$ and $C_2$ with trivializations given by the restrictions of $\phi$ and $\psi_p$ (for some choice of $p$ in the interval $(0,1)$), respectively. 
    Then, by the description of transition maps of $\mathfrak{p}$, we obtain that $V_j \times (C_1 \cap C_2)$ is invariant under $\phi^{-1} \circ \psi_p$ (and hence also under $\psi_p^{-1} \circ \phi$, as $\phi^{-1} \circ \psi_p$ is a homeomorphism), since the $V_j$'s are invariant under $f$ (and hence also under $f^{-1}$, as $f$ is a homeomorphism).
    Let $\{C_1^{\mathrm{o}},C_2^{\mathrm{o}}\}$ be the open cover of $S^1$, where $C_r^{\mathrm{o}}$ is the interior of $C_r$.
    Since $C_r$ is contractible, the chain of inclusions $C_r^{\mathfrak{o}} \subseteq C_r \subseteq S^1$ satisfies the condition that $C_r^{\mathfrak{o}}$ is categorical in $C_r$ and $C_r$ is categorical in $S^1$, see \Cref{rem: induced trivialization over contractible subset}.
    Thus, the hypotheses of \Cref{thm: additive upper bound for ls category} are satisfied since $\overline{C_r^{\mathrm{o}}} = C_r$, and we obtain $\ct(M_f) \leq 2 + n -1 = n+1$.
\end{proof}

\begin{theorem}
\label{thm: TC of mapping torus}
    Suppose $f \colon M \to M$ is a homeomorphism of a topological space $M$ and $M_f$ is the corresponding mapping torus. 
    If $(M_f)^k$ is completely normal and $\{V_1, \dots, V_n\}$ is an open cover of $M^k$ such that 
    \begin{enumerate}
        \item[(i)] each $V_j$ admits a continuous section of the free path space fibration $e_{k,M}$, and
        \item[(ii)] each $V_j$ is invariant under $f^k := f \times \dots \times f$,
    \end{enumerate}
    then $\TC_k(M_f) \leq k+n$.
\end{theorem}

\begin{proof}
    Let $\{R_i\}_{i=1}^{k+1}$ be the closed cover of $S^1$ defined in \Cref{prop: cat and tck of gpps with base S1}. 
    Then the fibre bundle $M \hookrightarrow M_f \xrightarrow{\mathfrak{p}} S^1$ is trivial over each $R_i$ with trivializations given by the restrictions of $\phi$ (resp. $\psi_p$ for some choice of $p \in S^1$) if $x_0 \not\in R_i$ (resp. $x_0 \in R_i)$.
    Let $\{U_i\}_{i=1}^{k+1}$ be the open cover of $(S^1)^k$ defined in \Cref{prop: cat and tck of gpps with base S1}.
    Then these covers satisfy the condition (i) of \Cref{additive inequality for TC for fibre bundles} for the fibre bundle $\mathfrak{p}$ as showed in \Cref{prop: cat and tck of gpps with base S1}.
    Moreover, by \Cref{remark for additive inequality of TC for fibre bundles}, the induced trivializations (see \Cref{def: induced trivialization for TC} and \Cref{additive inequality for TC for fibre bundles}) extend to local trivializations over $\overline{U}_i$.
    Then, by the description of transition maps of $\mathfrak{p}$, the transition maps of $\mathfrak{p}^k$ corresponding to the induced trivializations preserves the $V_j \times (\overline{U}_i \cap \overline{U}_{i'})$ since $V_j$'s are invariant under $f^k$ (and hence also $(f^{-1})^k$ since $f$ is a homeomorphism) and the induced trivializations are the restrictions of $\phi^k$ or $\psi_p^k$. 
    Hence, the condition (ii) of \Cref{additive inequality for TC for fibre bundles} is satisfied, and we get $\TC_k(M_f) \leq (k+1)+n-1 = k+n$.
\end{proof}

Note that the theorems above are more general than \Cref{prop: cat and tck of gpps with base S1}, in the sense that every fibre bundle over $S^1$ can be realized as a mapping torus of some homeomorphism of the fibre. For completeness, we state the next theorem, which provides a lower bound for the topological complexity of mapping tori defined using diffeomorphisms.
We thank the reviewer for suggesting that we include the fixed-point condition, which was missing from the original statement of the following theorem.


\begin{theorem}[{\cite[Corollary 2.8]{Mapping-theorems-for-TC}}]
    Let $\phi \colon M \to M$ be a diffeomorphism of a connected, smooth manifold $M$ having a fixed point, and let $M_\phi$ denote the mapping torus. Then $\TC(M_\phi) \geq \ct(M \times S^1)$.
\end{theorem}


\noindent\textbf{Acknowledgement}:
The author thanks the reviewer for insightful suggestions and feedback that improved the article—particularly for recommending the correct formulation of \Cref{lemma: induced trivialization for cat}, suggesting the inclusion of computations for projective product spaces, and other valuable contributions.
The authors thank Prof. John Oprea for pointing out that the fixed-point condition in \Cref{prop: TC_k-wrt-G leq TC_(k+1)-wrt-G} can be omitted, and for generously sharing his proof with us.
The first author would like to acknowledge IISER Pune - IDeaS Scholarship and Siemens-IISER Ph.D. fellowship for economical support. 
The second author acknowledges the support of NBHM through grant 0204/10/(16)/2023/R\&D-II/2789.

\bibliographystyle{plain} 
\bibliography{references}

\begin{thebibliography}{10}

\bibitem{Rudyak2014}
Ibai Basabe, Jes\'us Gonz\'alez, Yuli~B. Rudyak, and Dai Tamaki.
\newblock Higher topological complexity and its symmetrization.
\newblock {\em Algebr. Geom. Topol.}, 14(4):2103--2124, 2014.

\bibitem{BaySarkarheqtc}
Marzieh Bayeh and Soumen Sarkar.
\newblock Higher equivariant and invariant topological complexities.
\newblock {\em J. Homotopy Relat. Struct.}, 15(3-4):397--416, 2020.

\bibitem{secat}
I.~Berstein and T.~Ganea.
\newblock The category of a map and of a cohomology class.
\newblock {\em Fundam. Math.}, 50:265--279, 1962.

\bibitem{colmangranteqtc}
Hellen Colman and Mark Grant.
\newblock Equivariant topological complexity.
\newblock {\em Algebraic \& Geometric Topology}, 12(4):2299--2316, 2013.

\bibitem{CLOT}
Octav Cornea, Gregory Lupton, John Oprea, and Daniel Tanr\'{e}.
\newblock {\em Lusternik-{S}chnirelmann category}, volume 103 of {\em Mathematical Surveys and Monographs}.
\newblock American Mathematical Society, Providence, RI, 2003.

\bibitem{daundkarlens}
Navnath Daundkar.
\newblock Group actions and higher topological complexity of lens spaces.
\newblock {\em J. Appl. Comput. Topol.}, 8:2051–2067, 2024.

\bibitem{DaundSarkargpps}
Navnath Daundkar and Soumen Sarkar.
\newblock L{S}-category and topological complexity of several families of fibre bundles.
\newblock {\em Homology Homotopy Appl.}, 26(2):273--295, 2024.

\bibitem{Davis}
Donald~M. Davis.
\newblock Projective product spaces.
\newblock {\em J. Topol.}, 3(2):265--279, 2010.

\bibitem{n-Klein-bottle}
Donald~M. {Davis}.
\newblock {An \(n\)-dimensional Klein bottle}.
\newblock {\em {Proc. R. Soc. Edinb., Sect. A, Math.}}, 149(5):1207--1221, 2019.

\bibitem{Dold}
Albrecht Dold.
\newblock Erzeugende der {T}homschen {A}lgebra {${\mathcal N}$}.
\newblock {\em Math. Z.}, 65:25--35, 1956.

\bibitem{strongeqtc}
Alexander Dranishnikov.
\newblock On topological complexity of twisted products.
\newblock {\em Topology Appl.}, 179:74--80, 2015.

\bibitem{Fadelleqcat}
E.~Fadell.
\newblock The equivariant {L}justernik-{S}chnirelmann method for invariant functionals and relative cohomological index theories.
\newblock In {\em Topological methods in nonlinear analysis}, volume~95 of {\em S\'em. Math. Sup.}, pages 41--70. Presses Univ. Montr\'eal, Montreal, QC, 1985.

\bibitem{FarberTC}
Michael Farber.
\newblock Topological complexity of motion planning.
\newblock {\em Discrete Comput. Geom.}, 29(2):211--221, 2003.

\bibitem{Farbergrant}
Michael Farber and Mark Grant.
\newblock Robot motion planning, weights of cohomology classes, and cohomology operations.
\newblock {\em Proc. Amer. Math. Soc.}, 136(9):3339--3349, 2008.

\bibitem{F-O}
Michael Farber and John Oprea.
\newblock Sequential parametrized topological complexity and related invariants.
\newblock {\em Algebr. Geom. Topol.}, 24(3):1755--1780, 2024.

\bibitem{Vandembroucq}
Seher Fi\c{s}ekci and Lucile Vandembroucq.
\newblock On the {LS}-category and topological complexity of projective product spaces.
\newblock {\em J. Homotopy Relat. Struct.}, 16(4):769--780, 2021.

\bibitem{eqtcprodineq}
Jes\'us Gonz\'alez, Mark Grant, Enrique Torres-Giese, and Miguel Xicot\'encatl.
\newblock Topological complexity of motion planning in projective product spaces.
\newblock {\em Algebr. Geom. Topol.}, 13(2):1027--1047, 2013.

\bibitem{Grantfibrations}
Mark Grant.
\newblock Topological complexity, fibrations and symmetry.
\newblock {\em Topology Appl.}, 159(1):88--97, 2012.

\bibitem{grant2019symmetrized}
Mark Grant.
\newblock Symmetrized topological complexity.
\newblock {\em Journal of Topology and Analysis}, 11(02):387--403, 2019.

\bibitem{Mapping-theorems-for-TC}
Mark Grant, Gregory Lupton, and John Oprea.
\newblock A mapping theorem for topological complexity.
\newblock {\em Algebr. Geom. Topol.}, 15(3):1643--1666, 2015.

\bibitem{hatcher}
Allen Hatcher.
\newblock {\em Algebraic topology}.
\newblock Cambridge University Press, Cambridge, 2002.

\bibitem{eqlscategory}
Wac{\l}aw Marzantowicz.
\newblock A {$G$}-{L}usternik-{S}chnirelman category of space with an action of a compact {L}ie group.
\newblock {\em Topology}, 28(4):403--412, 1989.

\bibitem{Naskar}
Bikramaditya Naskar and Soumen Sarkar.
\newblock On {LS}-category and topological complexity of some fiber bundles and {D}old manifolds.
\newblock {\em Topology Appl.}, 284:107367, 14, 2020.

\bibitem{PaulSen}
Amit~Kumar Paul and Debasis Sen.
\newblock An upper bound for higher topological complexity and higher strongly equivariant complexity.
\newblock {\em Topology Appl.}, 277:107172, 18, 2020.

\bibitem{RUD2010}
Yuli~B. Rudyak.
\newblock On higher analogs of topological complexity.
\newblock {\em Topology Appl.}, 157(5):916--920, 2010.

\bibitem{sarkargpps}
Soumen Sarkar and Peter Zvengrowski.
\newblock On generalized projective product spaces and {D}old manifolds.
\newblock {\em Homology Homotopy Appl.}, 24(2):265--289, 2022.

\bibitem{Sva}
Albert~S. \v{S}varc.
\newblock The genus of a fiber space.
\newblock {\em Dokl. Akad. Nauk SSSR (N.S.)}, 119:219--222, 1958.

\end{thebibliography}

\end{document}